\documentclass[12pt]{article}
\usepackage{josPackage}
\usepackage{hyperref}

\title{The 1-periodic derived category of a gentle algebra :
	Part 1 - Indecomposable objects}
\author{Joseph Winspeare}
\date{}

\begin{document}

	\maketitle
	
	\section*{Abstract}
	
	Combining results from Keller and Buchweitz, we describe the 1-periodic derived category of a finite dimensional algebra $A$ of finite global dimension as the stable category of maximal Cohen-Macaulay modules over some Gorenstein algebra $A^\ltimes$. 
	In the case of gentle algebras, using the geometric model introduced by Opper, Plamondon and Schroll, we describe indecomposable objects in this category using homotopy classes of curves on a surface. In particular, we associate a family of indecompoable objects to each primitive closed curve. We then prove using results by Bondarenko and Drozd concerning a certain matrix problem, that this constitutes a complete description of indecomposable objects. 
	\tableofcontents
	
	\section*{Introduction} 
	
	The aim of this paper is to define the 1-periodic derived category of a gentle algebra of finite global dimension and describe the isomorphism classes of indecomposable objects using the geometric model in \cite{OPS}.
	
	\noindent Gentle algebras first appear in work by Assem and Happel \cite{AH} as tilted algebras of type $A$ and in work by Assem and Skowronski \cite{AS} as tilted algebras of type $\tilde{A}$. Their module category as well as their derived category have been well studied (\cite{BR,CB},\cite{BR,CB}) using string combinatorics. 
	 
	In particular, given a gentle algebra $A$, a complete explicit description of indecomposable objects in the derived category $\Dcat^b(A)$ has been given in \cite{BM} as well as an explicit description of morphism spaces in \cite{ALP}. These descriptions are combinatorial in nature and are given in terms of the quiver and relations associated to $A$. The class of gentle algebras is closed under derived equivalences \cite{SZ} and a first derived invariant was described in \cite{AAG} using these combinatorial techniques.
	 
	These results have been translated in geometrical terms in \cite{BCS} for the module category and in \cite{OPS} for the derived category.  More precisely, one can associate to $A$ a dissected marked surface $(S,M,P,\Delta)$. A link between the derived category and the partially wrapped Fukaya category of this surface was given in \cite{HKK}. Indecomposable objects are then described by homotopy classes of graded curves, morphisms are given by oriented intersections and certain cones can be described by resolutions of crossings. This model has given rise to a complete derived invariant for gentle algebras (see \cite{LP,APS,Opp,JSW}).
	 
	In the geometric model for the derived category, the homotopy class of a closed curve on $S$ gives rise to a family of indecomposable objects if and only if it is gradable, i.e. if its winding number with regards to the dissection $\Delta$ is zero. However, all homotopy classes of closed curves play an important role as they are necessary to the computation of the derived invariant mentionned above. This motivates the construction of a category in which all homotopy classes of closed curves give rise to a family of indecomposable objects.
	
	A natural candidate for this category is the 1-periodic derived category of the gentle algebra, that we define as the triangulated hull of the orbit category $\Dcat^b(A)/[1]$ in the sense of Keller \cite{Kel}. 
	
	The main result of this paper is the following 
	
	\begin{theorem*}[Theorem 3.3.6.]
		Let $A$ be a gentle algebra of finite global dimension with associated marked surface $(S,M,\Delta)$. The isomorphism classes of indecomposable objects in $(\Dcat^b(A)/[1])_\Delta$ are described by "band objects" and "string objects", where
		\begin{itemize}
			\item "string objects" are parametrized by non trivial homotopy classes of curves on $(S,M,\Delta)$ joining two marked points where the homotopy has fixed ends.
			\item "band objects" are parametrized by non trivial homotopy classes of primitive closed curves and by isomorphism classes of indecomposable $k[x,x^{-1}]$-modules.
		\end{itemize}
	\end{theorem*}

	Note that the category $(\Dcat^b(A)/[1])_\Delta$ has also been considered in \cite{Chr} as the 1-periodic Fukaya category of the surface $S$ and a description of indecomposable objects using curves on the surface has been given. His approach is however completely different to ours : the 1-periodic Fukaya category is defined as a global section of a (co)sheaf of stable $\infty$-categories in the sense of \cite{DK}.
	
	Here, we use Keller's approach to triangulated orbit categories which can be seen as a full subcategory of the singularity category of a certain DG-algebra. 
	 
	More precisely, we study in Section 1 this triangulated orbit category when $A$ is any finite-dimensional $k$-algebra of finite global dimension. This leads us to consider the Gorenstein algebra $A^\ltimes \simeq A \otimes_k k[\varepsilon]/\langle \varepsilon^2 \rangle$ (Proposition 1.2.4.), its singularity category $\Dcat_{\textup{sg}}(A^\ltimes)$ and its stable category of maximal Cohen-Macaulay modules $\barCM(A^\ltimes)$. 
	 We then obtain the following commutative diagram (Theorem 1.3.4.) of triangulated functors.
	 
	 \begin{center}
	 	\begin{tikzpicture}[scale = 1.2]
	 		\node (1) at (0,0) {$(\Dcat^b(A)/[1])_\Delta$};
	 		\node (2) at (0,2) {$\Dcat^b(A)$};
	 		\node (3) at (4,0) {$\barCM(A^\ltimes)$};
	 		\node (4) at (4,2) {$\barCM^\Z (A^\ltimes)$};
	 		\node (5) at (2,-2) {$\Dcat_{\textup{sg}}(A^\ltimes)$};
	 		
	 		\draw[->] (2) to node[above,midway] {$\sim$}(4);
	 		\draw[->] (1) to node[above,midway] {$\sim$} (3);
	 		\draw[->] (2) to (1);
	 		\draw[->] (4) to (3);
	 		\draw[->] (1) to node[below, midway, sloped] {\cite{Kel}} node[above,midway,sloped] {$\sim$} (5);
	 		\draw[->] (5) to node[below, midway, sloped] {\cite{Bu}} node[above,midway,sloped] {$\sim$} (3);
	 	\end{tikzpicture}
	 \end{center}
	 
	 Then a careful analysis of Buchweitz's equivalence leads to an explicit description of the essential image of the functor $\Dcat^b(A) \to (\Dcat^b(A)/[1])_\Delta$ (Corollary 1.3.9.). We further prove that this functor preserves almost split sequences (Proposition 1.4.1.).
	 	 
	 The case where $A$ is gentle is studied in sections 2 and 3. We give here an explicit description of objects associated to homotopy classes of curves on $S$ in terms of maximal Cohen-Macaulay $A^\ltimes$-modules. The main theorem is shown in Section 3. The proof is similar to the one in \cite{BM} and uses the Bondarenko-Drozd classification on indecomposable objects in a certain category of matrices (\cite{BD}).

	\section*{Conventions}
	
	In this paper, unless otherwise stated, $k$ is a field and $A$ is a finite dimensional $k$-algebra of finite global dimension. 
	
	\noindent If $B$ is a finite dimensional algebra, we denote $\mod \, B$ the category of finite dimensional right $B$-modules, $\Dcat^b(B)$ the bounded derived category of finite dimensional $B$-modules and $\per(B)$ the full subcategory of $\Dcat^b(B)$ consisting of objects quasi-isomorphic to a bounded complex of projective $B$-modules.
	
	\noindent If $Q$ is a quiver, we denote $Q_0$ the set of vertices and $Q_1$ the set of arrows. If $\alpha$ and $\beta$ are two paths of $Q$ such that the target of $\alpha$ is the source of $\beta$, we denote $\alpha \beta$ the path obtained by concatening them.	
	
	\section{The orbit category of $\Dcat^b(A)$ under $[1]$}
	
	\subsection{Orbit categories}
	
	In this section, we recall results from \cite{Kel} concerning orbit categories and apply them to $(\Dcat^b(A)/[1])_\Delta$.
	
	\begin{definition}
		Let $\Tcat$ be an additive category and $F$ an autoequivalence of $\Tcat$. We define the quotient category $\Tcat/F$ as the category with objects those of $\Tcat$ and morphism spaces :
		\[\forall X,Y \in \Tcat, \, \Hom_{\Tcat/F}(X,Y) = \bigoplus_{i \in \Z} \Hom_\Tcat (X,F^i Y)\]
	\end{definition}
	
	\begin{remark}
		The quotient category $\Tcat/F$ is an additive category and the natural projection functor $\Tcat \to \Tcat/F$ is an additive functor.
		
		However, if $\Tcat$ is a triangulated category and $F$ is an autoequivalence of triangulated categories then $\Tcat/F$ does not necessarily have a triangulated structure. 
	\end{remark}
	
	The definition of the orbit cateogry $(\Dcat^b(A)/[1])_\Delta$ can be found in \cite{Kel}. This definition comes with a fully faithful additive functor
	\[\Dcat^b(A)/[1] \longrightarrow (\Dcat^b(A)/[1])_\Delta\]
	such that the composition
	\[\Dcat^b(A) \longrightarrow \Dcat^b(A)/[1] \longrightarrow (\Dcat^b(A)/[1])_\Delta\]
	is a triangulated functor.
	
	\begin{definition}
		We define the trivial extension of $A$ with regards to itself as $A^{\ltimes} = A \oplus A$ with product :
		\[\forall a,a',x,x' \in A, \; (a,x).(a',x') = (aa', ax' + xa')\]
	\end{definition}
	
	\begin{theorem}{(\cite[Section 7]{Kel})}
		
		Let $A$ be as in the previous definition, then the projection on the first coordinate $A^\ltimes \twoheadrightarrow A$ induces a functor $\Dcat^b(A) \to \Dcat^b(A^{\ltimes})$. This functor in turn induces an equivalence of triangulated categories
		\[K : (\Dcat^b(A)/[1])_\Delta \simeq \Dcat_{sg}(A^\ltimes):= \Dcat^b(A^{\ltimes})/\per(A^{\ltimes})\]
	\end{theorem}
	
	\begin{proof}
		It is shown in \cite{Kel,Kelcor} that the functor $K$ exists and is fully faithful. It remains to show that $K$ is dense.
		
		By definition of $K$, the following diagram commutes :
		
		\begin{center}
			\begin{tikzpicture}[scale = 1.8]
				\node (1) at (0,0) {$\Dcat^b(A)$};
				\node (2) at (2,0) {$\Dcat^b(A^\ltimes)$};
				\node (3) at (4,0) {$\Dcat_{sg}(A^\ltimes)$};
				\node (5) at (2,-1) {$(\Dcat^b(A)/[1])_{\Delta}$};
				\draw[-to] (1) to node[above, midway] {$\iota_A$} (2);
				\draw[->>] (2) to node[above, midway] {$\pi_{A^\ltimes}$} (3);
				\draw[-to] (1) to (5);
				\draw[right hook-latex] (5) to node[below, midway] {$K$} (3);
			\end{tikzpicture}
		\end{center}
		
		Let $\Scat_{A^\ltimes}$ (resp. $\Scat_{A}$) be the full subcategory of $\Dcat^b(A^\ltimes)$ (resp. $\Dcat^b(A)$) consisting of simple objects. It is known that the thick subcategory of $\Dcat^b(A^\ltimes)$ generated by the simple objects $\textup{thick}(\Scat_{A^\ltimes})$ is $\Dcat^b(A^\ltimes)$. Because $\pi_{A^\ltimes}$ is essentially surjective, we have $\textup{thick}(\pi_{A^\ltimes}(\Scat_{A^\ltimes})) = \Dcat_{sg}(A^\ltimes)$. 
		
		However, $\iota_A$ restricts to an equivalence between $\Scat_{A}$ and $\Scat_{A^\ltimes}$, thus we have
		\[\Dcat_{sg}(A^\ltimes) = \textup{thick}(\pi_{A^\ltimes}\circ \iota_A(\Scat_A)).\]
		
		This finishes the proof.
	\end{proof}

	\subsection{Properties of $A^\ltimes$}
	
	\begin{definition}
		We define the following maps that will be useful in the rest of the paper :
		
		\begin{multicols}{2}
			{\begin{align*}
				\pi : A^\ltimes & \longrightarrow A \\
				(a,x) & \mapsto a
			\end{align*}}
			{\begin{align*}
				\iota : A & \longrightarrow A^\ltimes \\
				a & \mapsto (a,0)
			\end{align*}}
			{\begin{align*}
				\sigma : A^\ltimes & \longrightarrow A^\ltimes \\
				(a,x) & \mapsto (0,a)
			\end{align*}}
			{\begin{align*}
				q : A^\ltimes & \longrightarrow A^\ltimes \\
				(a,x) & \mapsto (a,-x)
			\end{align*}}
		\end{multicols}
		
		The map $\iota$ is a morphism of $A$-bimodules and the maps $\pi$, $\sigma$ and $q$ are morphisms of $A^\ltimes$-bimodules.
	\end{definition}
	
	We will now be considering algebras defined as quotients of path algebras of quivers. If $Q$ is a quiver, we define $Q_0$ to be the set of vertices and $Q_1$ the set of arrows of $Q$.
	
	\begin{proposition}
		If $A = kQ/I$, then $A^\ltimes = kQ'/I'$ where :
		\begin{itemize}
			\item $Q'_0 = Q_0$ ;
			\item $Q'_1 = Q_1 \cup \{\varepsilon_i : i \to i | i \in Q_0\}$
			\item the relations in $I'$ are those of $I$ to which we add the relations $\varepsilon_i ^2 = 0$ for all $i \in Q_0$ and $\varepsilon_i \alpha = \alpha \varepsilon_j$ for all $\alpha : i \to j $ in $Q_1$.
		\end{itemize}
	\end{proposition}
	
	\begin{proof}
		Let $kQ'/I'$ be the algebra defined above. Let $w$ be a path in $(Q',I')$, using the new relations we see that there is at most one arrow $\varepsilon_i$ in $w$. 
		
		We can now define the map :
		
		\begin{align*}
			f :  kQ'/I' &\to A^\ltimes \\
			w \, \textup{path of} \, (Q',I') &\mapsto \begin{cases}
				(0,w) \, \textup{if $w$ contain an arrow $\varepsilon_i$} \\
				(w,0) \, \textup{else}
			\end{cases}
		\end{align*}
		We now show that it is a well defined morphism of algebras. First, it is a linear map of vector spaces as it is defined on basis elements. If $w$ and $w'$ are two paths of $(Q',I')$ then by considering all possible cases, it is straightforward to show that $f(ww') = f(w)f(w')$.
		
		We thus have a morphism of algebras that is injective. However, the dichotomy of paths being in one of the two cases above shows that $\dim_k (kQ'/I') = 2. \dim_k (A) = \dim_k(A^\ltimes)$. Thus it is an isomorphism.
	\end{proof}
	
	\begin{remark}
		The previous description of $A^\ltimes$ also stands if $A$ is of infinite global dimension.
	\end{remark}
	
	\begin{proposition}
		There is an isomorphism of $k$-algebras :
		\[A^\ltimes \simeq A \otimes_k k[\varepsilon]/\langle \varepsilon^2 \rangle\]
	\end{proposition}
	
	\begin{proof}
		One easily checks that the following morphisms of $k$-algebras 
		\begin{multicols}{2}
			\begin{center}
			\begin{align*}
			f :  A^\ltimes &\to A \otimes_k k[\varepsilon]/\langle \varepsilon^2 \rangle \\
			(a,b) &\mapsto a \otimes 1 + b \otimes \varepsilon
			\end{align*}
			\end{center}
			\columnbreak 
			\begin{center}
			\begin{align*}
			g :  A \otimes_k k[\varepsilon]/\langle \varepsilon^2 \rangle &\to A^\ltimes \\
			m \otimes (\alpha + \beta \varepsilon)  &\mapsto (m \alpha , m \beta)
			\end{align*}
			\end{center}
		\end{multicols}
			
		are inverses of each other.
	\end{proof}
	
	\noindent Recall that a $k$-algebra $B$ is Gorenstein if $B$ has finite injective dimension as a right $B$-module and as a left $B$-module.
	
	\begin{corollary}
		The $k$-algebra $A^{\ltimes}$ is Gorenstein.
	\end{corollary}
	
	\begin{proof}
		The algebra $k[\varepsilon]/\langle \varepsilon^2 \rangle$ is a Gorenstein algebra. Therefore $A^\ltimes$ is Gorenstein if and only if $A$ is (\cite{AR}).
	\end{proof}
	
	Because $A^\ltimes$ is Gorenstein, we can consider maximal Cohen-Macaulay modules over $A^\ltimes$. We now recall from \cite{Bu} the definition of maximal Cohen-Macaulay modules.
		
	\begin{definition}
		Let $B$ be a finite dimensional Gorenstein algebra. A $B$-module $M$ is said to be maximal Cohen-Macaulay if we have :
		\[ \forall i \neq 0 , \, \Ext^i_B(M,B) = 0\]
		The full subcategory of $\mod \, B$ consisting of maximal Cohen-Macaulay modules is denoted $\textup{CM}(B)$.
	\end{definition}
	
	\begin{remark}
		The projective $B$-modules are maximal Cohen-Macaulay, thus by definition they are projective-injective in $\CM(B)$. This endows $\CM(B)$ with a structure of Frobenius category. Thus we can consider its stable category $\barCM(B)$ which has a structure of triangulated category (\cite{Hap}). 
	\end{remark}
	
	\begin{proposition}
		Let $M$ be a $A^\ltimes$-module. Then $M$ is maximal Cohen-Macaulay if and only if $M_A$ is a projective $A$-module. 
	\end{proposition}
	
	\begin{proof}
		As an $A$-module, $A^\ltimes$ is isomorphic to $A \oplus A$ and is thus projective. We can use \cite[Proposition~2.4.]{Ass} : 
		\[\forall n \geq 0, \; \forall X \in \mod A^\ltimes, \; \forall U \in \mod A, \; \Ext_{A^\ltimes}^n(X,\Hom({}_{A^\ltimes }A_A^\ltimes,U)) \simeq \Ext_A^n (X,U)\]
		If we use this isomorphism for $X=M$ a $A^\ltimes$-module and $U = A$, we get :
		\[\forall n \geq 0, \; \Ext^n_{A^\ltimes}(M,\Hom_A (A^\ltimes, A)) \simeq \Ext_A^n (M,A)\]
		We will now show that there is an isomorphism of $A^\ltimes$-modules between $\Hom_A (A^\ltimes, A)$ and $A^\ltimes$. Consider the following morphism of $A^\ltimes$-modules :
		\begin{align*}
			\Gamma : A^\ltimes & \longrightarrow \Hom_A(A^\ltimes, A) \\
			(a,x) & \mapsto (f_{(a,x)} : (b,y) \mapsto ay+xb)
		\end{align*}
		This morphism is additive, and if $(a,x),(a',x'),(b,y) \in A^\ltimes$ then :
		
		\[\left(\Gamma((a,x))._{A^\ltimes}(a',x')\right) (b,y) = f_{(a,x)}((a',x')(b,y)) = f_{(a,x)}(a'b, a'y + x'b) = aa'y+ax'b+xa'b\]
		
		and 
		\[\Gamma((a,x)(a',x'))(b,y) = \Gamma((aa',ax'+xa'))(b,y) = f_{(aa',ax'+xa')}(b,y) = aa'y+ax'b+xa'b\]
		Thus it is a morphism of $A^\ltimes$-modules. Consider the following morphism of $A^\ltimes$-modules :
		\begin{align*}
			\Delta : \Hom_A(A^\ltimes,A) & \longrightarrow A^\ltimes \\
			f &\mapsto (f(0,1),f(1,0))
		\end{align*}
		This map is clearly additive. To see that it is a morphism of $A^\ltimes$-modules, let $(a,x)$ be in $A^\ltimes$ and $f : A^\ltimes \to A$ be morphism of $A$-modules. We have :
		\begin{align*}
			\Delta(f.(a,x)) &= (f((a,x)(0,1)),f((a,x)(1,0))) \\
			&= (f(0,a),f(a,x)) \\
			&= (f(0,1)a, f(0,1)x + f(1,0)a) \\
			&= (f(0,1),f(1,0))(a,x) \\
			&= \Delta(f)(a,x)
		\end{align*}
		We can now observe that $\Delta$ and $\Gamma$ are mutual inverses, which proves :
		\[\Hom_A(A^\ltimes,A) \simeq A^\ltimes \]
		Thus $M$ is a maximal Cohen-Macaulay module over $A^\ltimes$, if and only if it is a maximal Cohen-Macaulay module over $A$. Since $A$ has finite global dimension, this second condition is equivalent to being a projective $A$-module.
	\end{proof}

	\begin{definition}
		Let $B$ be a finite dimensional $k$-algebra. We define the category of differential modules $\Dif \, B$ as follows :
		\begin{itemize}
			\item the objects are pairs $(M,d)$ where $M$ is a finitely generated $B$-module and $d$ is an endomorphism of $M$ such that $d^2=0$ ;
			\item for any two objects $(M,d)$ and $(N,d')$, a morphism $f : (M,d) \to (N,d')$ is the data of a morphism of $B$-modules $f : M \to N$ such that $f \circ d = d' \circ f$.
		\end{itemize}
		We also define $\Pcat(B)$ to be the full subcategory of $\Dif \, B$ consisting of objects $(P,d)$ where $P$ is a projective $B$-module.
	\end{definition}
	
	The categories $\Dif \, B$ and $\Pcat(B)$ are additive categories.
	
	\begin{proposition}
		Let $M$ be a $A^\ltimes$ module. We define $\Psi(M)$ to be the differential module $(M_A,d)$ where $M_A$ is the restriction of $M$ to $\mod \, A$ and $d$ is the endomorphism of $M_A$ given by the action of $\varepsilon$ on $M$ when $A^\ltimes$ is seen as the algebra $A \otimes_k k[\varepsilon]/\langle \varepsilon^2 \rangle$.
		
		If $f :M \to N$ is a morphism of $A^\ltimes$-modules, we define $\Psi(f)$ to be the restriction of $f$ to a morphism of $A$-modules $f_A : M_A \to N_A$.
		
		Then $\Psi$ defines an additive functor $\mod \, A^\ltimes \to \Dif \, A$ that is an equivalence of categories. Furthermore, this functor restricts to an equivalence of categories $\CM(A^\ltimes) \overset{\sim}{\to} \Pcat(A)$.
	\end{proposition}
	
	\begin{proof}
		To show that it is indeed an equivalence of categories, we will give its quasi-inverse. We define the functor $\Phi : \Dif \, A \to \mod \, A^\ltimes$ as follows :
		\begin{itemize}
			\item if $(M,d)$ is an object of $\Dif \, A$, we define the $A^\ltimes$-module as having underlying vector space $M$ and the action of $A^\ltimes \simeq A \otimes_k k[\epsilon]/ \langle \epsilon^2 \rangle$ as follows :
			\[\forall a\otimes f \in A \otimes_k k[\epsilon]/ \langle \epsilon^2 \rangle \; \forall x \in M, \; x.(a\otimes f) = a f(d)(x)\]
			where $f$ in $k[\varepsilon]/\langle \varepsilon^2 \rangle$ is the polynomial of degree at most one representing its equivalence class.
			\item the functor $\Phi$ is defined to be the identity on morphisms.
		\end{itemize}
		Thus defined, $\Phi$ is an additive functor and we can check that it is a quasi-inverse to $\Psi$. 
		
		Using Proposition 1.2.8., we see that $\Psi$ restricts to an equivalence of categories  $$\CM(A^\ltimes) \overset{\sim}{\to} \Pcat(A)$$.
	\end{proof}
	
	This equivalence endows $\Pcat(A)$ with a structure of Frobenius category in which the projective-injective objects are the pairs $(P\oplus P, \begin{pmatrix}
		0 & \id_P \\
		0 & 0
	\end{pmatrix})$.
	
	Recall that a graded $k$-algebra $B$ is a $k$-algebra with the data of a decomposition $B = \bigoplus_{i \in \Z} B_i$ such that if $x$ and $y$ are respectively in $B_i$ and $B_j$ then $xy$ is in $B_{i+j}$. A graded (right) $B$-module is a $B$-module $M = \bigoplus_{i \in \Z} M_i$ such that if $x$ and $m$ are respectively in $B_i$ and $M_j$ then $m.x$ is in $M_{i+j}$. A morphism of graded $B$-modules $f : M \to N$ is a morphism of $B$-modules such that for all $i$ in $\Z$, $f(M_i) \subset N_i$. If $M$ is a graded $B$-module, we denote $M(1)$ the shifted module with $(M(1))_i = M_{i-1}$. We denote $\mod^\Z \, B$ the category of graded $B$-modules.
	
	We now endow $A^\ltimes$ with a grading. To do so, recall the isomorphism $A^\ltimes \simeq A \otimes_k k[\varepsilon]/\langle \varepsilon^2 \rangle$ from Proposition 1.2.4. By considering $A$ as a graded algebra concentrated in degree zero and the grading on $k[\varepsilon]/\langle \varepsilon^2 \rangle$ by setting $\deg(\varepsilon)=1$, we have an induced grading on $A^\ltimes$. As $A$ is a subalgebra of $A^\ltimes$, it has an induced grading which is the trivial grading.
	
	With this grading, we can consider $\CM^\Z(A^\ltimes)$ the full subcategory of $\mod^\Z \, A^\ltimes$ consisting of graded maximal Cohen-Macaulay modules. 
	
	\begin{definition}
		We define the category $\Dif^\Z(A)$ as follows :
		\begin{itemize}
			\item objects are pairs $(M,d)$ where $M$ is a graded $A$-module and $d$ is an endomorphism of $M$ with degree 1.
			\item for any two objects $(M,d)$ and $(N,d')$, a morphism $f : (M,d) \to (N,d')$ is the data of a degree zero morphism of graded $A$-modules $f :M \to N$ such that $f \circ d = d' \circ f$.
		\end{itemize}
		We also define $\Pcat^\Z (A)$ to be the full subcategory of $\Dif^\Z (A)$ consisting of objects $(P,d)$, where $P$ is a projective $A$-module.
	\end{definition}
	
	Using the same arguments as previously, there is an equivalence of categories $\Pcat^\Z (A) \simeq \CM^\Z(A)$ that endows $\Pcat^\Z(A)$ with a structure of Frobenius category.	

	\begin{proposition}
		There exists a commutative diagram 
		
		\begin{center}
			\begin{tikzpicture}[scale = 2]
				\node (1) at (0,0) {$\Ccat^b(\proj \, A)$};
				\node (2) at (3,0) {$\Pcat^\Z(A) \simeq \CM^\Z (A^\ltimes)$};
				\node (3) at (0,-1) {$\Dcat^b(A) \simeq \Kcat^b(\proj \, A)$};
				\node (4) at (3,-1) {$\underline{\mathcal{P}}^\Z(A) \simeq \barCM^\Z(A^\ltimes)$};
				
				\draw[->] (1) to node[above, midway] {$\tilde{\Gamma}$} (2);
				\draw[->] (1) to (3);
				\draw[->] (3) to node[above, midway] {$\Gamma$} (4);
				\draw[->] (2) to (4);
			\end{tikzpicture}
		\end{center}
		
		\noindent where $\tilde{\Gamma}$ is an exact equivalence of categories and $\Gamma$ is a triangulated equivalence of categories induced by $\tilde{\Gamma}$.
	\end{proposition}
	
	\begin{proof}
		Let $\tilde{\Gamma} : \Ccat^b(\proj \, A) \to \Pcat^\Z(A)$ be the functor defined as follows :
		\begin{itemize}
			\item If $(P^\bullet,d^\bullet)$ is a bounded complex of projective $A$-modules then $\tilde{\Gamma}((P^\bullet,d^\bullet)) = (\bigoplus_{i \in \Z} P^i, \sum_{i \in \Z} d^i)$ where the grading on $\bigoplus_{i \in \Z} P^i$ as an $A^\ltimes$-module is given by the degree of each $P^i$.
			\item A morphism of complexes $f^\bullet : P^\bullet \to Q^\bullet $ is sent to $\tilde{\Gamma}(f^\bullet) = \sum_{i \in \Z} f^i$.
		\end{itemize} 
		
		The functor $\tilde{\Gamma}$ is well defined on objects. To see that it is well defined on morphisms, let $f^\bullet : P^\bullet \to Q^\bullet$ be a morphism of cochain complexes of projective $A$-modules. Then the map
		\begin{align*}
			\tilde{\Gamma}(f) : \bigoplus_{i \in \Z} P^i &\to \bigoplus_{i \in \Z} Q^i \\
			(x_i)_{i \in \Z} &\mapsto (f^i(x_i))_{i \in \Z}
		\end{align*}
		is a morphism of $A^\ltimes$-modules. Indeed, if $(a,b)$ is in $A^\ltimes$, we have :
		\begin{align*}
			\tilde{\Gamma}(f)((x_i)_{i \in \Z}.(a,b)) &= \tilde{\Gamma}(f)((x_i a + d^{i-1}(x_{i-1}b))_{i \in \Z}) \\
			&= (f^i(x_i a + d^{i-1}(x_{i-1}b)))_{i \in \Z} \\
			&= (f^i(x_i)a + d^{i-1}(f^{i-1}(x_{i-1})b))_{i \in \Z} \\
			&= (f^i(x_i))_{i \in \Z} .(a,b) \\
			&= \tilde{\Gamma}(f)((x_i)_{i \in \Z}).(a,b)
		\end{align*}
		
		One easily checks that it is an additive equivalence of categories. Furthermore, if $0 \to (A^\bullet,d_A^\bullet) \overset{f^\bullet}{\to} (B^\bullet,d_B^\bullet) \overset{g^\bullet}{\to} (C^\bullet,d_C^\bullet) \to 0$, its image by $\tilde{\Gamma}$ is the short exact sequence
		\[0 \to (\bigoplus_{i \in \Z} A^i , \sum_{i \in \Z} d_A^i) \overset{\sum_{i \in \Z} f^i}{\longrightarrow}(\bigoplus_{i \in \Z} B^i , \sum_{i \in \Z} d_B^i) \overset{\sum_{i \in \Z} g^i}{\longrightarrow} (\bigoplus_{i \in \Z} C^i , \sum_{i \in \Z} d_C^i) \to 0\]
		This equivalence is thus an exact equivalence between two Frobenius categories. The projective-injective objects in $\Ccat^b(\proj A)$ are isomorphic to complexes of the form $\dots \to 0 \to P \xlongequal{} P \to 0 \to \dots$. We now see that $\tilde{\Gamma}$ sends projective-injective in $\Ccat^b(\proj A)$ to projective-injective objects in $\Pcat^\Z(A)$.
		Using \cite[Lemma p23]{Hap}, the functor $\tilde{\Gamma}$ restricts to a triangulated equivalence $\Gamma : \Kcat^b(\proj A) \to \underline{\mathcal{P}}^\Z (A)$.
		\end{proof}
		
		\begin{remark}
			The functors $\tilde{\Gamma} \circ [1]$ and $(1) \circ \tilde{\Gamma}$ are not isomorphic as the functor $[1]$ introduces a sign on the differential whereas $(1)$ does not change the graded module structure. 
		\end{remark}
	
	\begin{proposition}
		Let $M$ be a maximal Cohen-Macaulay $A^\ltimes$-module. There exists a unique decomposition $M \simeq M' \oplus M_{\textup{proj}}$ where $M_{\textup{proj}}$ is a projective $A^\ltimes$ module and $M'$ is of the form $(Q,\psi)$ with $\Im(\psi) \subset \textup{rad}(Q)$.
	\end{proposition}
	
	\begin{proof}
		We write $M = (P,\varphi)$ with $P$ a projective $A$-module and $\varphi$ an endomorphism of $P$ such that $\varphi^2 = 0$. This data is equivalent to that of a 1-periodic cochain complex of projective $A$-modules, let us denote it $P^\bullet$. As an object of $\Ccat(\proj \, A)$, we have a unique decomposition $P^\bullet = R^\bullet \oplus Q^\bullet $ where $Q^\bullet$ is the direct sum of cochain complexes of the form $\dots\to 0 \to X = X \to 0 \to \dots$ and $R^\bullet$ has no direct summands of this form.
		
		Because $P^\bullet$ is 1-periodic, so are $R^\bullet$ and $Q^\bullet$. This gives a decomposition $(P,\varphi) = (R,\psi) \oplus (Q,\psi')$.
		
		If a cochain complex of the form $\dots\to 0 \to X = X \to 0 \to \dots$ is a summand of $Q^\bullet$ then so are all its shifts. We can thus construct a summand of $Q^\bullet$
		\[\dots \to X\oplus X \overset{\begin{pmatrix}
				0 & id_X \\ 0 & 0
		\end{pmatrix}}{\to} X\oplus X \overset{\begin{pmatrix}
				0 & id_X \\ 0 & 0
		\end{pmatrix}}{\to} X \oplus X \to \dots  \; .\]
		
		We now see that $Q^\bullet$ can be written as a finite direct sum of cochain complexes of this form. Thus it is a cochain complex of the form : $(X \oplus X, \begin{pmatrix}
			0 & id_X \\
			0 & 0
		\end{pmatrix})$ , i.e. the projective $A^\ltimes$-module $X \otimes_A A^\ltimes$.
		
		It remains to show that $\Im(\psi)$ is a subset of $\textup{rad}(R)$. To do so, we show that if it is not the case, then $(R,\psi)$ has a projective direct summand.
		
		Let $R = \bigoplus_{i \in I} R_i$ be the decomposition of $R$ into indecomposable projective $A$-modules.
		
		If there exists $x$ in $R$ such that $\psi(x)$ is not in $\textup{rad}(R)$ then there exists $k$ in $I$ such that $\pi_k\psi(x)$ is not in $\textup{rad}(R_k)$ where $\pi_k$ is the projection of $R$ onto $R_k$.
		
		Furthermore, if $x = \sum_{i \in I} x_i$ then there exists $l$ in $I$ such that $\pi_k \psi(x_l)$ is not in $\textup{rad}(R_k)$. This means that $f := \pi_k \psi_{|R_l} : R_l \to R_k$ is a morphism between indecomposable projective modules such that $\Im(f)$ is not a subset of $\textup{rad}(R_k)$. Thus $f$ is an isomorphism.
		
		If $k \neq l$, we write $R = R_k \oplus R_l \oplus R'$ and we consider the following morphism 
		\[(R_l \oplus R_l, \begin{pmatrix}
			0 & 0 \\ 1 & 0
		\end{pmatrix}) \overset{\begin{pmatrix}
				1 & \psi : R_l \to R_l \\
				0 & f \\
				0 & \psi :  R_l \to R'
		\end{pmatrix}}{\xrightarrow{\hspace*{6cm}}} (R,\psi)\]
		This morphism is well defined and is a monomorphism because $f$ is an isomorphism. However, $(R,\psi)$ is a Cohen-Macaulay $A^\ltimes$-module and $(R_l \oplus R_l, \begin{pmatrix}
			0 & 0 \\ 1 & 0
		\end{pmatrix})$ is projective-injective in $\CM(A^\ltimes)$. Thus it is a summand of $(R,\psi)$.
		
		If $k=l$ then by writing $R = R_l \oplus R'$, the restriction of $\varphi$ to $R_l$ is an isomorphism. Thus if $x$ is in $R_l \backslash \textup{rad}(R_l)$ then $\varphi(x)$ is not in $\textup{rad}(R_l)$.
	\end{proof}
	
	We can now define $\CMrad(A^\ltimes)$ as the full subcategory of $\textup{CM}(A^\ltimes)$ containing objects $(P,\varphi)$ such that $\Im(\varphi)$ is a subset of $\textup{rad}(P)$. By the previous proposition, we have a bijection between indecomposable objects of $\CMrad(A^\ltimes)$ and indecomposable objects of $\barCM(A^\ltimes)$.
	
	\subsection{Keller-Buchweitz functors}
	
	As seen in Corollary 1.2.5, $A^\ltimes$ is a Gorenstein algebra. When this is the case, there exists an equivalence of categories $\Dcat_{sg}(A^\ltimes) \overset{\sim}{\longrightarrow} \barCM(A^\ltimes)$ explicited in \cite{Bu}.
	
	The aim of this section is to show that this triangulated equivalence is compatible with the triangulated equivalence $(\Dcat^b(A)/[1])_\Delta \simeq \Dcat_{sg}(A^\ltimes)$ given by Keller (Theorem 1.1.4., \cite{Kel}) and the triangulated equivalence $\Dcat^b(A) \simeq \barCM^\Z(A^\ltimes)$ of Proposition 1.2.12.
	
	We now recall results from \cite{Bu} about maximal Cohen-Macaulay modules over a Gorenstein algebra $B$.
	
	\begin{itemize}
		\item If $M$ is a $B$-module, a complete resolution of $M$ is an acyclic complex of projective $B$-modules $(P^\bullet,d^\bullet)$ such that $\Coker(d^{-1} : P^{-1} \to P^0) = M$.
		\item $\underbar{APC}(B)$ denotes the homotopy category of acyclic complexes of finitely generated projective $B$-modules. It is a full triangulated subcategory of $\Kcat(\proj \, B)$ the homotopy category of cochain complexes of projective $B$-modules.
		\item For $i \in \Z$, we define the i-th syzygy functor as follows : 
		\[\forall (X^\bullet,d^\bullet) \in \Kcat(\proj \, B), \; \Omega_i (X^\bullet) = \Coker(d^{i-1} : X^{i-1} \to X^i)\]
	\end{itemize}
	
		If $M$ is a $B$-module and $p_M : P \to M$ a projective cover of $M$. We define 
		\[\Omega_B (M) = \Ker (p_M).\]
		This defines an autoequivalence of $\underbar{$\mod$} \, B$ called the loop-space functor.

		The i-th syzygy functor $\Omega_i$ defines a functor $\underbar{\APC}(B) \to \underbar{\mod} \, B$. This functor maps the inverse of the shift to the loop-space functor i.e. :
		\[\forall X \in \underbar{\APC}(B), \; \Omega_B(\Omega_i(X)) \simeq \Omega_i(X[-1])\]
	
		The composition of canonical functors 
		\[ \mod \, B \to \Dcat^b(B) \to \Dcat_{sg}(B)\]
		induces a unique functor $\iota_B : \underbar{\mod} \, B \to \Dcat_{sg}(B)$ that sends the loop-space functor to the inverse of the shift :
		\[\forall M \in \underbar{\mod} \, B, \; \iota_B(\Omega_B(M)) \simeq \iota_B(M)[-1]\]
	
	\begin{remark}
		When applying $\Omega_0$ to a object $P^\bullet$ in $\underbar{APC}(B)$, the $B$-module that is obtained is maximal Cohen-Macaulay because for all positive integers $i$ and $j$ with $j$ large enough $\Ext^i(\Omega_0 P^\bullet,B) = \Ext^{i+j}(\Omega_{-j}P^\bullet,B) = 0$. 
	\end{remark}
	
	We now have the following theorem giving us the desired equivalence.
	
	\begin{theorem}{\cite{Bu}}
		Let $B$ be a Gorenstein algebra. Then :
		\begin{itemize}
			\item The syzygy functor $\Omega_0$ induces a triangulated equivalence of categories, denoted by the same symbol :
			\[\Omega_0 : \underbar{\APC(B)} \to \barCM(B).\]
			\item The restriction of $\iota_B$ to $\barCM(B)$ defines a triangulated equivalence of categories :
			\[\iota_B : \barCM(B) \to \Dcat_{sg}(B).\]
		\end{itemize}
	\end{theorem}
	
	\begin{remark}
		These equivalences can be visualized on the following diagram :
		\begin{center}
			\begin{tikzpicture}[scale = 1.2]
				\node (1) at (0,4) {$0$};
				\node (2) at (2,4) {$\proj \, B$};
				\node (3) at (4,4) {CM(B)};
				\node (4) at (8,4) {$\barCM(B)$};
				\node (5) at (10,4) {$0$};
				\node (6) at (-2,2) {$0$};
				\node (7) at (0,2) {$\proj \, B$};
				\node (8) at (2,2) {$\mod \, B$};
				\node (9) at (5,2) {$\underbar{\mod} \, B$};
				\node (10) at (6.5,2) {$0$};
				\node (11) at (10,2) {$\underbar{\APC}(B)$};
				\node (12) at (0,0) {$0$};
				\node (13) at (2,0) {$\per \, B$};
				\node (14) at (4,0) {$\Dcat^b(B)$};
				\node (15) at (8,0) {$\Dcat_{sg}(B)$};
				\node (16) at (10,0) {$0$};
				
				\draw[-to] (1) to (2);
				\draw[-to] (2) to (3);
				\draw[-to] (3) to (4);
				\draw[-to] (4) to (5);
				\draw[-to] (6) to (7);
				\draw[-to] (7) to (8);
				\draw[-to] (8) to (9);
				\draw[-to] (9) to (10);
				\draw[-to] (12) to (13);
				\draw[-to] (13) to (14);
				\draw[-to] (14) to (15);
				\draw[-to] (15) to (16);
				\draw[-to] (2) to (7);
				\draw[-to] (3) to (8);
				\draw[-to] (4) to (9);
				\draw[-to] (4) to node[right,midway] {$\iota_B$} (15);
				\draw[-to] (11) to node[right,midway] {$\Omega_0$}(4);
				\draw[-to] (7) to (13);
				\draw[-to] (8) to (14);
				\draw[-to] (9) to node[left,midway] {$\iota_B$} (15);
				\draw[-to] (11) to node[right, midway] {$\iota_B \circ \Omega_0$}(15);
			\end{tikzpicture}
		\end{center}
	\end{remark}
	
	\begin{remark}
		The (quasi)-inverse of $\iota_B : \barCM(B) \to \Dcat_{sg}(B)$ is computed as follows :
		\begin{itemize}
			\item For a complex $X$ of $\Dcat^b(B)$, choose a projective resolution $P \to X$.
			\item Truncate the complex $P$ in degree $k$ where $k$ is an integer satisfying 
			\[k \leq \min \{ i \in \Z | H^i(P) = H^i(X) \neq 0 \} - \textup{injdim} \, R\]
			to obtain a complex $\sigma_{\leq k}(P)$ such that $\Omega_k(\sigma_{\leq k}(X))$ is maximal Cohen-Macaulay.
			\item Extend $\sigma_{\leq k}(P)$ to an acyclic complex of projective modules $\sigma_{\leq k}(P)^\#$.
			\item Take the 0th syzygy of this acyclic complex to obtain :
			\[\iota_B^{-1}(X) = \Omega_0(\sigma_{\leq k}(P)^\#)\]
		\end{itemize}
		This computation also gives the (quasi)-inverse of $\iota_B \circ \Omega_0$
	\end{remark}
	
	The aim of the rest of this section is now to prove the following result.
	
	\begin{theorem}
		The following diagram commutes.
		
		\begin{center}
			\begin{tikzpicture}[scale = 1.2]
				\node (1) at (0,0) {$\Dcat_{sg}(A^\ltimes)$};
				\node (2) at (0,2) {$\Dcat^b(A)$};
				\node (3) at (4,0) {$\barCM(A^\ltimes)$};
				\node (4) at (4,2) {$\barCM^\Z (A^\ltimes)$};
				
				\draw[->] (2) to node[above,midway] {$\Gamma$}(4);
				\draw[->] (1) to node[above,midway] {$\iota_{A^\ltimes}^{-1}$} (3);
				\draw[->] (2) to node[left, midway] {$K$} (1);
				\draw[->] (4) to node[right, midway] {$\mathcal{F}$} (3);
			\end{tikzpicture}
		\end{center}
		where 
		\begin{itemize}
			\item $\mathcal{F}$ is the functor "forgetting" the grading, induced by the functor $\grmod \, A^\ltimes \to \mod \, A^\ltimes$ ;
			\item $\Gamma$ is the functor from Proposition 1.2.12 ;
			\item $K$ is the triangulated functor given by Theorem 1.1.4 ;
			\item $\iota_{A^\ltimes}^{-1}$ is the equivalence of cateogories given above.
		\end{itemize} 
	\end{theorem}
	
	The proof of this theorem requires only to show that $F := \iota_{A^\ltimes}^{-1} \circ K$ and ${\mathcal{F}} \circ \Gamma$ are the same functors. We already have an explicit description of ${\mathcal{F}} \circ \Gamma$. We will now give one for $F$.
	
	\begin{proposition}
		Let $P$ be a projective $A$-module seen as an object of $\Dcat^b(A)$ concentrated in degree 0. Its image by the functor $K$ in $\Dcat_{sg}(A^{\ltimes})$ is isomorphic to :
		\[\dots \to P \otimes A^\ltimes \to P \otimes A^\ltimes \to P \otimes A^\ltimes \to 0 \to \dots\]
		where the differential is $id_P \otimes \sigma$.
	\end{proposition}
	
	\begin{proof}
		Let $P$ be a projective $A$-module. The functor $\Dcat^b(A) \to \Dcat_{sg}(A^\ltimes)$ is by definition induced by the functor $\Dcat^b(A) \to \Dcat^b(A^\ltimes)$ which is in turn induced by the functor $\mod \, A \to \mod \, A^\ltimes$. Thus the image of the complex $\dots \to 0 \to P \to 0 \to \dots$ in $\Dcat^b(A^\ltimes)/ \per \, A^\ltimes$ is :
		\[\dots \to 0 \to P_{|A^\ltimes} \to 0 \to \dots\]
		We can now compute its projective resolution by computing the projective resolution of $P_{|A^\ltimes}$ in $\mod \, A^\ltimes$. By observing that $P_{|A^\ltimes} = P \otimes_A A_{|A^\ltimes} = P \otimes_A \pi(A^\ltimes)$, we have a surjection:
		\[P \otimes_A A^\ltimes \overset{\id_P \otimes \pi}{\longrightarrow} P_{|A^\ltimes} \]
		We now see that this surjection fits into a short exact sequence :
		\[0 \to P_{|A^\ltimes} \overset{\id_P \otimes (\sigma \circ \iota)}{\longrightarrow} P \otimes_A A^\ltimes \overset{\id_P \otimes \pi}{\longrightarrow} P_{|A^\ltimes} \to 0\]
		where $\iota : A \to A^\ltimes $ is defined by $\iota(x) = (x,0)$.
		Thus we can complete the projective resolution of $P_{|A^\ltimes}$ and we have the desired isomorphism in $\Dcat_{sg} (A^\ltimes)$.
	\end{proof}
	
	\begin{corollary}
		Let $P$ be a projective $A$-module, then $F(P) = P_{|A^\ltimes}$ where $F = \iota_{A^\ltimes}^{-1} \circ K$ is as in Theorem 1.3.5.
	\end{corollary}
	
	\begin{proof}
		It is enough to observe that $\Coker(\id_P \otimes \sigma) \simeq P_{|A^\ltimes}$.
	\end{proof}
	
	Using this description of the image of projective $A$-modules, we can now infer the corresponding result for any bounded complex of projectives.
	
	\begin{proposition}
		Let $P^\bullet$ be a bounded complexe of projective $A$-modules :
		\[P^\bullet = \dots \to 0 \to P^{-n} \overset{d^{-n}}{\to} P^{-(n-1)} \to \dots \to P^{-1} \overset{d^{-1}}{\to} P^0 \to 0 \to \dots\]
		Then its image by in $\Dcat_{sg}(A^\ltimes)$ is isomorphic to :
		\[\dots \to \bigoplus_{i=0}^n P^{-i} \otimes A^\ltimes \overset{\varphi}{\to} \bigoplus_{i=0}^n P^{-i} \otimes A^\ltimes  \overset{\varphi^n}{\to} \bigoplus_{i=0}^{n-1} P^{-i} \otimes A^\ltimes \to \dots \to P^{-1} \otimes A^\ltimes \oplus P^0\otimes A^\ltimes \overset{\varphi^1}{\to} P^0\otimes A^\ltimes \to 0 \to \dots \]
		where :
		\begin{itemize}
			\item for $i \in \{1,\dots,n\}$, $\varphi^i = \displaystyle\sum_{j=0}^{i-1} id_{P^{-j}}\otimes \sigma + \displaystyle\sum_{j=0}^i d^{-j} \otimes q$
			\item $P^0 \otimes A^\ltimes$ is in the same degree as $P^0$ in $P^\bullet$.
			\item $\varphi = \displaystyle\sum_{j=0}^n id_{P^j}\otimes \sigma + \displaystyle\sum_{j=0}^n d^j \otimes q$
		\end{itemize}
		Furthermore, $F(P^\bullet) \simeq \Coker(\varphi)$ in $\barCM(A^\ltimes)$.
	\end{proposition}
	
	\begin{proof}
		To compute the image of $P^\bullet$ in $\Dcat_{sg}(A^\ltimes)$, we need to compute its projective resolution. To do so, we use the projective resolution of the $(P^{-i})_{0 \leq i \leq n}$ in $\mod \, A^\ltimes$ given in Proposition 1.3.6.
		
		\noindent Because $\sigma \circ q + q \circ \sigma = 0$, we get the following double complex :
		\begin{center}
			\begin{tikzpicture}[scale = 1.5]
				\node (1) at (-0.5,0) {$0$};
				\node (2) at (1,0) {$P^{-n}$};
				\node (3) at (3,0) {$\dots$};
				\node (4) at (5,0) {$P^{-1}$};
				\node (5) at (8,0) {$P^0$};
				\node (6) at (9.5,0) {$0$};
				\node (8) at (-0.5,2) {$0$};
				\node (9) at (1,2) {$P^{-n} \otimes_A A^\ltimes$};
				\node (10) at (3,2) {$\dots$};
				\node (11) at (5,2) {$P^{-1} \otimes_A A^\ltimes$};
				\node (12) at (8,2) {$P^0 \otimes_A A^\ltimes$};
				\node (13) at (9.5,2) {$0$};
				\node (15) at (-0.5,4) {$0$};
				\node (16) at (1,4) {$P^{-n} \otimes_A A^\ltimes$};
				\node (17) at (3,4) {$\dots$};
				\node (18) at (5,4) {$P^{-1} \otimes_A A^\ltimes$};
				\node (19) at (8,4) {$P^0 \otimes_A A^\ltimes$};
				\node (20) at (9.5,4) {$0$};
				\node (21) at (1,5) {$\vdots$};
				\node (22) at (5,5) {$\vdots$};
				\node (23) at (8,5) {$\vdots$};
				\node (24) at (1,-1) {$0$};
				\node (25) at (5,-1) {$0$};
				\node (26) at (8,-1) {$0$};
				\draw[-to] (1) to (2);
				\draw[-to] (2) to node [above,midway] {$d^{-n}$}(3);
				\draw[-to] (3) to node [above,midway] {$d^{-2}$}(4);
				\draw[-to] (4) to node [above,midway] {$d^{-1}$} (5);
				\draw[-to] (5) to (6);
				\draw[-to] (8) to (9);
				\draw[-to] (9) to node [above, midway] {$d^{-n} \otimes q$} (10);
				\draw[-to] (10) to node [above, midway] {$d^{-2} \otimes q$}(11);
				\draw[-to] (11) to node [above,midway] {$d^{-1} \otimes q$} (12);
				\draw[-to] (12) to (13);
				\draw[-to] (15) to (16);
				\draw[-to] (16) to node [above, midway] {$d^{-n} \otimes q$} (17);
				\draw[-to] (17) to node [above,midway] {$d^{-2} \otimes q$} (18);
				\draw[-to] (18) to node [above,midway] {$d^{-1} \otimes q$}(19);
				\draw[-to] (19) to (20);
				\draw[-to] (21) to (16);
				\draw[-to] (22) to (18);
				\draw[-to] (23) to (19);
				\draw[-to] (16) to node [right,midway] {$\id \otimes \sigma$} (9);
				\draw[-to] (18) to node [right,midway] {$\id \otimes \sigma$} (11);
				\draw[-to] (19) to node [right,midway] {$\id \otimes \sigma$} (12);
				\draw[-to] (9) to node [right,midway] {$\id \otimes \pi$}(2);
				\draw[-to] (11) to node [right,midway] {$\id \otimes \pi$} (4);
				\draw[-to] (12) to node [right,midway] {$\id \otimes \pi$} (5);
				\draw[-to] (2) to (24);
				\draw[-to] (4) to (25);
				\draw[-to] (5) to (26);
			\end{tikzpicture}
		\end{center}
		Because the columns are exact, its total complex is acyclic, let us denote it $C^\bullet$.
		\\
		Consider the double complex defined as the one above, except set the $(P^{-i})_{0 \leq i \leq n}$ to zero as well as the relevant maps. Consider its total complex $R^\bullet$. This is the complex that we want to show is isomorphic to $P^\bullet$ in $\Dcat_{sg}^b (A^\ltimes)$. 
		\\
		The cochain map $f^\bullet : R^\bullet \to P^\bullet$ defined by $f^{-i} = \id_{P^{-i}}\otimes \pi$ if $0 \leq i \leq n$ and $f^i = 0$ otherwise has mapping cone $C^\bullet$. Therefore it is a quasi-isomorphism because $C^\bullet$ is acyclic.
	\end{proof}
	
	\begin{corollary}
		Using the notations of the previous proposition, if $P^\bullet$ has no projective-injective summands in $\Ccat^b(\proj \, A)$ we have $\dim_k(F(P^\bullet)) = \frac{1}{2}\dim_k (\bigoplus_{i=0}^n P^i \otimes_A A^\ltimes)$.
	\end{corollary}
	
	\begin{proof}
		By definition, we have $F(P^\bullet) = \Coker(\varphi)$. However, we have $\Im(\varphi) = \Ker(\varphi)$ thus $\dim_k(\Ker(\varphi)) = \frac{1}{2}\dim_k (\bigoplus_{i=0}^n P^i \otimes_A A^\ltimes)$. 
		\\
		The result follows from 
		\[\dim_k(\Coker(\varphi)) = \dim_k(\bigoplus_{i=0}^n P^i \otimes_A A^\ltimes) - \dim_k (\Im(\varphi))\]
	\end{proof}
	
	\begin{corollary}
		Using the notations of the previous proposition, if $(P^\bullet, d^\bullet)$ is a bounded complex of projective $A$-modules with no projective-injective summands in $\Ccat^b(\proj A)$, then:
		\[F(P^\bullet) \simeq (\bigoplus_{i \in \Z} P^i, \sum_{i \in \Z} d^i)\]
	\end{corollary}
	
	\begin{proof}
		Let $N$ be the module given by the pair $(\bigoplus_{i \in \Z} P^i, \displaystyle\sum_{i \in \Z} d^i)$. We know that $F(P^\bullet) \simeq \Coker(\varphi)$ with :
		\[\varphi = \displaystyle\sum_{j=0}^n id_{P^j}\otimes \sigma + \displaystyle\sum_{j=0}^n d^j \otimes q\]
		\noindent Let $\pi_2 : A^\ltimes \to A$ be the morphism of $A$-bimodules defined by $\pi_2(a,x) = x$.
		
		\noindent Now consider the morphism $\delta : \bigoplus_{i=0}^n P^i \otimes_A A^\ltimes \to N$ defined by 
		\[\delta = \id \otimes \pi + \sum_{i=0}^n d^i \otimes \pi_2 \]
		This map is surjective as it has a section $\delta^* := \id \otimes \iota - \displaystyle\sum_{i=0}^n d^i \otimes (\sigma \circ \iota)$. Furthermore, as  $\delta \circ \varphi = 0$ we have a surjection $F(P^\bullet) \to N$.
		
		\noindent By construction, we have $\dim_k(F(P^\bullet)) = \frac{1}{2}\dim_k(\bigoplus_{i=0}^n P^i \otimes_A A^\ltimes) = \dim_k (N)$ which means that the surjection $F(P^\bullet) \to N$ is an isomorphism.
	\end{proof}
	
	\begin{remark}
		From this corollary, we see that the 1-periodic derived category that we consider is the same as the derived category of 1-periodic complexes as given in \cite{Sta} and the functor $F$ is the compression functor that is also considered.
	\end{remark}
	
	\begin{proof}(Theorem 1.3.5)
		From Proposition 1.2.12 and Corollary 1.3.12., we get that the diagram commutes for objects. It remains to show that it commutes on morphisms.
		Let $f^\bullet : P^\bullet \to Q^\bullet$ be a morphism in $\Kcat^b(\proj \, A)$. Then $K(f^\bullet)$ is the morphism of complexes defined by 
		\[K(f^\bullet)^{-j} = \sum_{i=0}^{j} f^{_i} \otimes \id_{A^\ltimes} : \bigoplus_{i=0}^j P^{-i} \otimes A^\ltimes \to \bigoplus_{i=0}^j Q^{-i} \otimes A^\ltimes. \]
		If $j$ is greater than $\textup{gldim}(A)$, we have 
		\[K(f^\bullet)^{-j} = \sum_{i=0}^{n} f^{_i} \otimes \id_{A^\ltimes} : \bigoplus_{i=0}^n P^{-i} \otimes A^\ltimes \to \bigoplus_{i=0}^n Q^{-i} \otimes A^\ltimes.\]
		Thus $\iota_{A^\ltimes}^{-1} \circ K(f^\bullet) = \displaystyle\sum_{i=0}^n f^{-i} = \Fcat \circ \Gamma(f^\bullet)$.
	\end{proof}
	
	\begin{remark}
		The suspension functor in $\barCM(A^\ltimes)$ is not isomorphic to the identity functor (see \cite{Kelcor}). 
		If $(P,\psi)$ is an object of $\CM^{A^\ltimes}$, the following short exact sequence computes its syzygy 
		\[0 \to (P,-\psi) \overset{\begin{pmatrix}
				\psi \\
				- \id
		\end{pmatrix}}{\longrightarrow} (P \oplus P, \begin{pmatrix}
			0 & 0 \\
			\id & 0
		\end{pmatrix}) \overset{\begin{pmatrix}
			\id & \psi
			\end{pmatrix}}{\longrightarrow} (P,\psi) \to 0\]
		Thus the suspension $\Omega$ is not always the identity functor but $\Omega^2$ is the identity. 
	\end{remark}
	
	\begin{remark}
		The category $\barCM(A^\ltimes)$ can be seen as a special case of a $Q$-shaped derived category as defined in \cite{HJ}. Using notations from this paper, it is the case where $Q$ is the cyclic quiver with one vertex (see \cite[Figure 2]{HJ}). Here, $\mod \, A^\ltimes$ plays the role of ${}_{Q,A} \Mod$ (as a consequence of Proposition 1.2.10). Furthermore, one may show that $\add(A^\ltimes)$ is the subcategory of exact objects $\mathcal{E}$. Thus $\CM(A^\ltimes)$ is analog to ${}^\perp \mathcal{E}$ and $\barCM(A^\ltimes)$ to ${}^\perp \mathcal{E} / {}_{Q,A} \textup{Proj}$.
	\end{remark}
	
	\subsection{Auslander-Reiten triangles}
	
	Because $A^\ltimes$ is a Gorenstein algebra (\cite{AR}) the category $\CM(A^\ltimes)$ has Auslander-Reiten sequences (\cite[Section 3]{AR}). Therefore, the triangulated category $\barCM(A^\ltimes)$ has Auslander-Reiten triangles.
	
	However, the Auslander-Reiten sequences in $\CM(A^\ltimes)$ are different than the ones in $\mod \, A^\ltimes$. Indeed if $M$ is maximal Cohen-Macaulay, it is not necessarily true for $\tau M$. It is possible to compute the Auslander-Reiten translate in $\CM(A^\ltimes)$ from the one in $\mod \, A^\ltimes$ as shown in \cite{AR} but this is not the approach we take here.
	
	We want to prove the following result.
	
	\begin{proposition}
		The triangulated functor $\Dcat^b(A) \to \barCM(A^\ltimes)$ preserves Auslander-Reiten triangles.
	\end{proposition}
	
	To do so, we use the following more general result, inspired by \cite{GG}.

	\begin{lemma}
		Let $\Acat$ and $\Bcat$ be two Frobenius categories that have almost split exact sequences and $F : \Acat \to \Bcat$ be a fully faithful exact functor sending projective-injective objects of $\Acat$ to projective-injective objects of $\Bcat$. Then the following are equivalent :
		\begin{itemize}
			\item $F$ sends almost split sequences to almost split sequences ;
			\item For each almost split sequence $0 \to X \to Y \to Z \to 0$, if $X$ or $Z$ is isomorphic to an object of $\Im(F)$ then this short exact sequence is isomorphic to a short exact sequence of $\Im(F)$.
		\end{itemize}
	\end{lemma}
	
	\begin{proof}
		Assume that $F$ preserves almost split sequences. Let $0 \to X \to Y \to Z \to 0$ be an almost split sequence of $\Bcat$ such that $X$ is isomorphic to an object of $\Im(F)$. 
		
		There exists $X'$ in $\Acat$ that is not projective-injective such that $X \simeq FX'$. Consider the almost split sequence $0 \to X' \to Y' \to Z' \to 0$ associated with $X'$. As $F$ preserves almost split sequences, $0 \to FX' \to FY' \to FZ' \to 0$ is almost split in $\Bcat$. 
		
		Thus by uniqueness of almost split sequences, $0 \to X \to Y \to Z \to 0$ is isomorphic to this exact sequence of $\Im(F)$.
		
		The other case where $Z$ is a non projective-injective object of $\Bcat$ that is isomorphic to an object of $\Im(F)$ is dealt with in the same way.
		
		Conversely, let us assume that the second condition is true for $F$. We now show that $F$ preserves almost split sequences.
		
		Let $X'$ be an indecomposable non projective-injective object of $\Acat$. By assumption, $X := FX'$ is not projective-injective. Consider the almost split sequence $0 \to X \to Y \to Z \to 0$ in $\Bcat$.
		
		Because $X \in \Im(F)$, by hypothesis this almost split sequence is in $\Im(F)$, that is it can be written :
		\[0 \to FX' \overset{F(f)}{\longrightarrow} FY' \overset{F(g)}{\longrightarrow} FZ' \to 0\]
		
		We can then consider the sequence $0 \to X' \overset{f}{\to} Y' \overset{g}{\to} Z' \to 0$. It remains to show that it is an almost split sequence. 
		
		As $F$ is exact and fully faithful, this sequence is indeed exact. 
		
		Now, consider $v$ a non split epimorphism in $\Hom_\Acat (W,Z')$. As $F$ is exact and fully faithful, $F(v)$ is a non split epimorphism in $\Hom_\Bcat(FW,Z)$. Thus by definition of almost split sequences, there exists $s$ in $\Hom_\Bcat(FW,Y)$ such that $F(v) = F(g)s$.
		
		However, because $F$ is fully faithful, there exists $s'$ in $\Hom_\Acat(W,Y')$ such that $s = F(s')$ and thus $v = gs'$. This shows that $g$ is right almost split.
		
		We use analogous arguments to show that $f$ is left almost split.
	\end{proof}
	
	Before we prove Proposition 1.4.1., we recall the following classical result.
	
	\begin{lemma}(\cite{Rog})
		Let $\Ccat$ be a Frobenius category and $\underline{\Ccat}$ be its stable category. A short exact sequence of $\Ccat$
		\[0 \to X\to Y \to Z \to 0\]
		is almost split if and only if the associated triangle in $\underline{\Ccat}$
		\[X \to Y \to Z \to \Omega^{-1} X\]
		is almost split.
	\end{lemma}
	
	We now prove Proposition 1.4.1..
	
	\begin{proof}(Proposition 1.4.1.)
		As $\Gamma$ (defined in Theorem 1.3.7) is an equivalence of triangulated categories, it preserves Auslander-Reiten triangles. It is thus sufficient to prove that $\Fcat : \barCM^\Z(A^\ltimes) \to \barCM(A^\ltimes)$ preserves Auslander-Reiten triangles.
		
		To show that it preserves Auslander-Reiten triangles, it is sufficient by Lemma 1.4.3. to show that the functor $\tilde{\Fcat} : \CM^\Z(A^\ltimes) \to \CM(A^\ltimes)$ preserves almost split sequences. This functor is the one induced by the functor $\mod^\Z \, A^\ltimes \to \mod \, A^\ltimes$ that forgets grading.
		
		To do so, we use Lemma 1.4.2. as $\CM^\Z(A^\ltimes)$ and $\CM(A^\ltimes)$ both have almost split sequences (\cite{AR}). 
		
		The functor $\tilde{\Fcat}$ is a fully faithful exact functor because it is the case for the functor $\mod^\Z \, A^\ltimes \to \mod \, A^\ltimes$ from which it is induced. 
		
		Secondly, the projective-injective obejcts of $\CM^\Z(A^\ltimes)$ and $\CM(A^\ltimes)$ are the projective $A^\ltimes$-modules and thus are preserved by $\tilde{\Fcat}$.
		
		It remains to show that if $0 \to X \to Y \to Z \to 0$ is an almost split sequence of $\CM(A^\ltimes)$ such that $X$ or $Z$ is gradable then the three objects are gradable. 
		
		Let $\lambda$ be in $k^\times$, consider the following isomorphism of algebras :
		\begin{center}
			\begin{align*}
				t_\lambda : A^\ltimes &\to A^\ltimes \\
				(a,x) &\mapsto (a,\lambda x)
			\end{align*}
		\end{center}
		This isomorphism induces an autoequivalence of $\mod \, A^\ltimes$ that we will also write $t_\lambda$. If $X$ is an $A^\ltimes$-module, we define :
		\[t(X) := \{t_\lambda(X) | \lambda \in k^\times\}/ \simeq \]
		i.e. $t(X)$ is the set of isomorphism classes of $t_\lambda(X)$ for all $\lambda$ in $k^\times$.
		
		We know from \cite{AO14} (Theorem 4.2.) that $X$ is gradable if and only if $t(X)$ is finite.
		
		Let $0 \to X \to Y \to Z \to 0$ be an almost split sequence of $\CM(A^\ltimes)$ such that $X$ or $Z$ is in the image of $\tilde{\Fcat}$. Now, consider the set $T$ of isomorphism classes of short exact sequences $0 \to t_\lambda(X) \to t_\lambda(Y) \to t_\lambda(Z) \to 0$ for $\lambda$ in $k^\times$. Because all of the $t_\lambda$ are autoequivalences, they preserve almost split sequences and thus by uniqueness of short exact sequences up to isomorphism, if $t(X)$ or $t(Z)$ is finite then $T$ is finite. 
		
		Therefore, $t(X)$, $t(Y)$ and $t(Z)$ are all three finite i.e. $X$, $Y$ and $Z$ are gradable. Because $\tilde{\Fcat}$ is fully faithful, the short exact sequence $0 \to X \to Y \to Z \to 0$ is in the image of $\tilde{\Fcat}$.
		
		Using Lemma 1.4.2, we deduce that $\tilde{\Fcat}$ preserves almost split sequences. Therefore $\Fcat : \barCM^\Z(A^\ltimes) \to \barCM(A^\ltimes)$ preserves Auslander-Reiten triangles.
		
	\end{proof}
	
	\section{Explicit description for gentle algebras}
	
	From now on and until the end of the paper, $A$ is a gentle algebra of finite global dimension
	
	\subsection{Surface associated to a gentle algebra}
	
	In this section, we will recall the geometric model for gentle algebras in the case of algebras of finite global dimension. We first recall the definition of a gentle algebra.
	
	\begin{definition}
		A finite dimensional $k$-algebra $A$ is gentle if $A$ is isomorphic to $kQ/I$ where $Q$ is a finite connected quiver and $I$ is an ideal of the path algebra $kQ$ such that the pair $(Q,I)$ verify the following assumptions :
		\begin{itemize}
			\item for every vertex $i \in Q_0$, there is at most two arrows of target $i$ and at most two arrows of source $i$ ;
			\item the ideal $I$ is generated by paths of length 2 ;
			\item for every arrow $\alpha : i \to j$ there is at most one arrow $\beta$ of source $j$ such that $\alpha\beta \notin I$ and at most one arrow $\beta'$ of source $j$ such that $\alpha\beta' \in I$ ;
			\item for every arrow $\alpha : i \to j$ there is at most one arrow $\gamma$ of target $i$ such that $\gamma\alpha \notin I$ and at most one arrow $\gamma'$ of target $i$ such that $\gamma'\alpha \in I$.
		\end{itemize}
	\end{definition}
	
	\begin{remark}
		One of the immediate consequences of this definition is that in the quiver of a gentle algebra, every arrow is in a unique maximal path.
	\end{remark}
	
	\begin{definition}
		Let $A = kQ/I$ be a gentle algebra. We define :
		\begin{itemize}
			\item $\Mcat$ the set consisting of maximal paths of $(Q,I)$ and the trivial paths $e_i$ for the $i \in Q_0$ that are in only one maximal path.
			\item for $w \in \Mcat$ a path of length $n$, we define the $(n+2)$-gon $P_w$ as follows : if $v_0,\dots ,v_n$ are the vertices that lie on $w$ in traversal order, label the edges of $P_w$ in a counterclockwise fashion with the $v_i$. Then put a $\green$ vertex on the remaining edge. 
		\end{itemize}
	\end{definition}
	
	
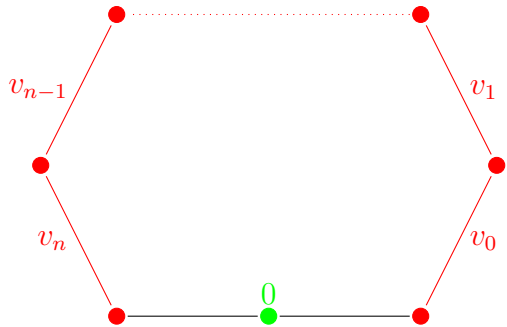
\begin{center}
		\begin{tikzpicture}
			\filldraw[red] (-2,0) circle (3pt) node (1) {$ $};
			\filldraw[red] (2,0) circle (3pt) node (2) {$ $};
			\filldraw[red] (3,2) circle (3pt) node (3) {$ $};
			\filldraw[red] (2,4) circle (3pt) node (4) {$ $};
			\filldraw[red] (-2,4) circle (3pt) node (5) {$ $};
			\filldraw[red] (-3,2) circle (3pt) node (6) {$ $};
			\filldraw[green] (0,0) circle (3pt) node (7) {$ $};
			\node[green] (8) at (0,0.3) {$0$};
			
			\draw (1) to (7);
			\draw (7) to (2);
			\draw[red] (2) to node[right, midway] {$v_0$} (3);
			\draw[red] (3) to node[right, midway] {$v_1$} (4);
			\draw[red,dotted] (4) to (5);
			\draw[red] (5) to node[left, midway] {$v_{n-1}$} (6);
			\draw[red] (6) to node[left, midway] {$v_n$} (1);
		\end{tikzpicture}
	\end{center}
	
	Each vertex of $Q$ appears on exactly two edges of the polygons $(P_w)_{w \in \Mcat}$. Thus, we can define the oriented surface with boundary $S$ by gluing the polygons along the edges with regards to the labels.
	
	The $\green$ vertices on $S$ are called marked points and lie on the boundary of $S$. The set of marked points is denoted by $M$. The edges of the polygons form a dissection of the surface denoted $\Delta$.
	
	\begin{remark}
		Here, $A$ is not necessarily of finite global dimension. However, it is the case if and only if each $\red$ vertex lies on the boundary of $S$.
	\end{remark}
	
	Using this construction, we have defined a triplet $(S,M,\Delta)$ for every gentle algebra. 
	
	\begin{remark}
		An arc $\gamma$ in a polygon $P_w$ joining two red arcs of the dissection always corresponds to a subpath of $w$ and thus to a morphism of $A$-modules between the corresponding projectives.
	\end{remark}
	
	\begin{example}
		Consider the following algebra $A$ which we will be using as an example throughout this paper :
		
		\begin{center}
			\begin{tikzpicture}
				\node (1) at (0,0) {$1$};
				\node (2) at (2,0){$2$};
				\node (3) at (4,0) {$3$};
				
				\draw[->] (1.20) to node[above, midway] {$a$} (2.160);
				\draw[->] (1.340) to node[below, midway] {$a'$} (2.200);
				\draw[->] (2.20) to node[above, midway] {$b$} (3.160); 
				\draw[->] (2.340) to node[below, midway] {$b'$} (3.200);
				\draw[dotted, bend left=90] (1.5,0.1) to (2.5,0.1);
				\draw[dotted, bend right=90] (1.5,-0.1) to (2.5,-0.1);
			\end{tikzpicture}
		\end{center}
		
		There are two maximal paths $ab'$ and $a'b$. The surface is thus obtained by gluing together two squares. Furthermore, we can see it is a torus with one boundary component. The marked dissected surface is as follows :
		
		
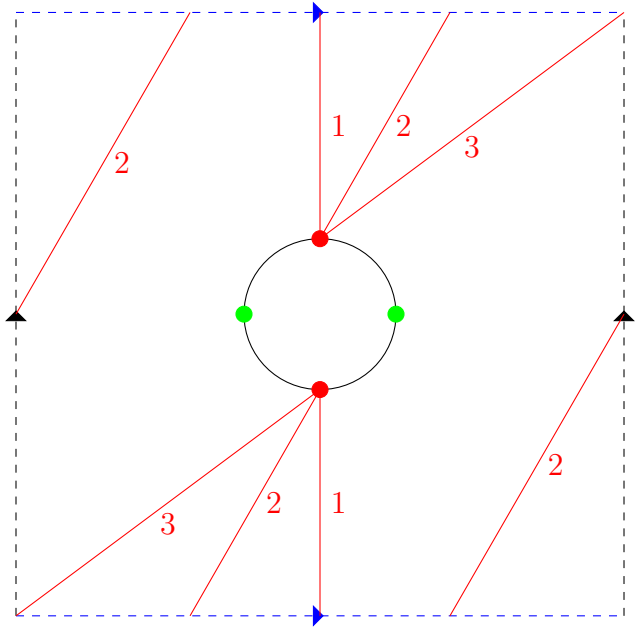
\begin{center}
			\begin{tikzpicture} 
				\begin{scope}[on background layer]
					
					\draw (4,4) circle (1);
				\end{scope}
				
				\begin{scope}[dashed, every node/.style={sloped,allow upside down}]
					\draw[blue] (0,0)-- node {\midarrow} (8,0);
					\draw[black] (0,0)-- node {\midarrow} (0,8);
					\draw[blue] (0,8)-- node {\midarrow} (8,8);
					\draw[black] (8,0)-- node {\midarrow} (8,8);
				\end{scope}
				
				\begin{scope}[on above layer]
					\filldraw[red] (4,3) circle (3pt) node (1) {$ $};
					\filldraw[red] (4,5) circle (3pt) node (2) {$ $};
					\filldraw[green] (3,4) circle (3pt) node (3) {$ $};
					\filldraw[green] (5,4) circle (3pt) node (4) {$ $};

					\draw[red] (4,5) to node[right, midway] {$1$} (4,8);
					\draw[red] (4,3) to node[right, midway] {$1$} (4,0);
					
					\draw[red] (4,5) to node[below, midway] {$3$} (8,8);
					\draw[red] (4,3) to node[below, midway] {$3$} (0,0);
					
					\draw[red] (4,5) to node[right, midway] {$2$} (5.71,8);
					\draw[red] (5.71,0) to node[right, midway] {$2$} (8,4);
					\draw[red] (0,4) to node[right, midway] {$2$} (2.29,8);
					\draw[red] (2.29,0) to node[right, midway] {$2$} (4,3);
				\end{scope}
			\end{tikzpicture}
		\end{center}
	\end{example}

	\subsection{Strings on the surface of a gentle algebra}
	
	\begin{definition}
		Let $(S,M,\Delta)$ be a marked dissected surface. A string on $(S,M,\Delta)$ is a homotopy class of curves joining two fixed marked points. We ask this curve to be non-homotopic to a trivial curve starting and ending at the same marked point. We denote $St(S,M)$ the set of all strings.
	\end{definition}
	
	\begin{remark}
		When considering a string $\gamma$, we will always consider a representative that is said to be in minimal position, meaning that it intersects the arcs of $\Delta$ in a minimal number of points and that these intersections are transverse. 
	\end{remark}
	
	\begin{remark}
		By definition, a string is not oriented. However, we may choose an orientation of a string. If $\overset{\to}{\gamma}$ is an oriented string, we denote $-\overset{\to}{\gamma}$ the same string with the reverse orientation.
	\end{remark}
	
	\begin{definition}
		Let $\overset{\to}{\gamma}$ be an oriented string of $(S,M,\Delta)$. Let $p_0, \dots ,p_n$ be the ordered intersection points between $\overset{\to}{\gamma}$ and $\Delta$. Consider the decomposition $\overset{\to}{\gamma} = \gamma_0 \gamma_1 \dots \gamma_n \gamma_{n+1}$ where :
		\begin{itemize}
			\item $\gamma_0$ is an arc between the starting point of $\overset{\to}{\gamma}$ and $p_0$ ;
			\item for $1 \leq i \leq n$, $\gamma_i$ is an arc between $p_{i-1}$ and $p_i$ ;
			\item $\gamma_{n+1}$ is an arc between $p_n$ and the end point of $\overset{\to}{\gamma}$ 
		\end{itemize}
		Because every arc of $\Delta$ is tagged with a vertex of the quiver of the gentle algebra, we can consider for $0 \leq i \leq n$ the vertex $v_i$ corresponding to $p_i$. 
		
		We can now associate to $\overset{\to}{\gamma}$ an object $(P,\varphi)$ of \underbar{$\Pcat$}$(A)$ and denote $M_{\overset{\to}{\gamma}}$ its image in $\barCM(A^\ltimes)$. This object is defined as follows :
		\begin{itemize}
			\item $P = \bigoplus_{i=0}^n P_{v_i}$ ;
			\item $\varphi = \sum_{i=0}^n f_i$ where $f_i : P_{v_{i-1}} \to P_{v_i}$ or $f_i : P_{v_i} \to P_{v_{i-1}}$ is the morphism defined by the arc $\gamma_i$ 
		\end{itemize}
	\end{definition}
	
	\begin{remark}
		We can see that by construction, if $\overset{\to}{\gamma}$ is an oriented string then $M_{\overset{\to}{\gamma}} = M_{-\overset{\to}{\gamma}}$. Thus we can define this object for a non oriented string of $\gamma$ in $St(S,M)$ and we will write it $M_\gamma$.
	\end{remark}
	
	We have a similar construction of a string object $P_{(\gamma,\mu)}^\bullet$ in $\Dcat^b(A)$ given in \cite{OPS}, where $\mu$ is a grading of $\gamma$. We first recall the definition of a grading on a string.
	
	\begin{definition}
		A grading on a string $\gamma$ (in minimal position) is a map $\mu : \gamma \cap \Delta \to \Z$ such that if $p$ and $q$ are two consecutive points of $\gamma \cap \Delta$ such that if $\tilde{\gamma}$ is an orientation of the arc of $\gamma$ joining $p$ to $q$ in a polygon defining the surface, we have :
		\[f(q) = \begin{cases}
			\mu(p) + 1 \; \text{if $\tilde{\gamma}$ has the boundary on its left} \\
			\mu(p) - 1 \; \text{if $\tilde{\gamma}$ has the boundary on its right} \\
		\end{cases}\]
	\end{definition}
	
	\begin{remark}
		Any two gradings on a string always differ by an integer as the grading is entirely defined by its value on one intersection point.
	\end{remark}
	
	\begin{definition}
		Let $\gamma$ be a string of $(S,M,\Delta)$ and $\mu$ a grading on $\gamma$. Consider $\overset{\to}{\gamma}$ an orientation of $\gamma$. Using the same notations as in definition 2.2.4., we define the complex $P^\bullet_{(\gamma,\mu)}$ as follows :
		\begin{itemize}
			\item $\forall n \in \Z, \, P^n_{(\gamma,\mu)} = \displaystyle\bigoplus_{\mu(p_i) = n} P_{v_i}$ ; 
			\item the differential $d^\bullet$ is given by the maps $f_i$ defined as in Definition 2.2.4..
		\end{itemize}
		This construction does not depend on the choice of orientation, thus we will not indicate one.
	\end{definition}
	
	\begin{proposition}
		Let $\gamma$ be a string of $(S,M,\Delta)$ and $\mu$ a grading of $\gamma$. Then $M_\gamma$ is isomorphic to $F(P^\bullet_{(\gamma,\mu)})$ where $F : \Dcat^b(A) \to \barCM(A^\ltimes)$ is the functor defined in Theorem 1.3.5.
	\end{proposition}
	
	\begin{proof}
		Using Corollary 1.3.12.,the definition of $P^\bullet_{(\gamma,\mu)}$ and of $M_\gamma$, we see that these objects are indeed isomorphic.
	\end{proof}
	
	\subsection{Bands on the surface of a gentle algebra}
	
	\begin{definition}
		Let $(S,M)$ be a marked surface, a band on $(S,M)$ is a non trivial homotopy class of a primitive closed curve. We denote $Ba(S,M)$ the set of all bands on $(S,M)$.
	\end{definition}
	
	\begin{remark}
		As with strings, we consider oriented bands. Furthermore, if $\overset{\to}{\gamma}$ is an oriented band and $p_0$ is an intersection point of $\gamma$ with $\Delta$, then there is a unique decomposition $\overset{\to}{\gamma} = \gamma_1 \dots \gamma_r$ where the sarting point of $\gamma_1$ is $p_0$ and each $\gamma_i$ is an arc of $\overset{\to}{\gamma}$ joining two consecutive intersection points of $\gamma$ with $\Delta$.
		
		For $1 \leq i \leq r$, we denote $p_i$ the end point of $\gamma_i$ and thus $p_0 = p_n$. We will also consider $v_i$ the vertex on the quiver of $A$ corresponding to the arc of $\Delta$ on which  $p_i$ sits.
	\end{remark}
	
	\begin{remark}
		Here, the $k[x,x^{-1}]$-module $(k^n,J)$ is the one with underlying vector space $k^n$ and where the action of $x$ is given by the automorphism $J$. 
	\end{remark}
	
	\begin{definition}
		Let $(k^n,J)$ be an indecomposable $k[x,x^{-1}]$-module. Using the previous notations, we define the object $M_{(\overset{\to}{\gamma},p_0,J)}$ in $\barCM(A^\ltimes)$ as follows :
		\begin{itemize}
			\item $P = \bigoplus_{i = 0}^{r-1} P_{v_i}^{\oplus n}$
			\item $\varphi = \sum_{i=1}^{r-1} f_i I_n + f_r J$ where  $f_i : P_{v_i} \to P_{v_{i+1}}$ or $f_i : P_{v_{i+1}} \to P_{v_i}$ is the morphism given by $\gamma_i$.
			\item $M_{(\overset{\to}{\gamma},p_0,J)}$ is the image in $\barCM(A^\ltimes)$ of $(P,\varphi)$ by the equivalence of Corollary 1.2.10.
		\end{itemize}
	\end{definition}
	
	\begin{remark}
		We also define a grading on bands in the same way as we defined a grading on a string. However, not all bands have a grading. As with strings, two gradings on a band always differ by an integer.
	\end{remark}
	
	\begin{definition}
		Using the previous notations, let $\overset{\to}{\gamma}$ be an oriented band, $p_0$ a starting point, $\mu$ a grading of $\gamma$ and $(k^n, J)$ an indecomposable $k[x,x^{-1}]$-module. We define the band complex $P^\bullet_{(\overset{\to}{\gamma},p_0,\mu,J)}$ as follows :
		\begin{itemize}
			\item $\forall j \in \Z, \; P^j_{(\overset{\to}{\gamma},p_0,\mu,J)} = \displaystyle\bigoplus_{\mu(p_i) = j} P^{\oplus n}_{v_i}$
			\item the differential $d^\bullet$ is given by the maps $f_i I_n$ for $1 \leq i \leq r-1$ and the map $f_r J$ where the $(f_i)_{1 \leq i \leq r}$ are the maps defined by the $(\gamma_i)_{1 \leq i \leq r}$.
		\end{itemize}
	\end{definition}
	
	\begin{proposition}
		Let $\overset{\to}{\gamma}, p_0, \mu, J$ be respectively :
		\begin{itemize}
			\item an oriented band of $(S,M,\Delta)$, 
			\item a starting point of $\gamma$,
			\item a grading of $\gamma$,
			\item an endomorphism of $k^n$ defining an indecomposable $k[x,x^{-1}]$-module $(k^n,J)$.
		\end{itemize}
		Then we have an isomorphism
		\[M_{(\overset{\to}{\gamma},p_0,J)} \simeq F(P^\bullet_{(\overset{\to}{\gamma},p_0,\mu,J)})\]
		where $F : \Dcat^b(A) \to \barCM(A^\ltimes)$ is the functor defined in THeorem 1.3.5.
	\end{proposition}
	
	\begin{proof}
		The proof follows from Corollary 1.3.12. and the two definitions above.
	\end{proof}
	
	\begin{remark}
		If we consider the set of parameters $(\overset{\to}{\gamma}, p_0, J)$ as in definition 2.3.3. that we will denote $P(S,M,\Delta)$, then two distinct parameters do not necessarily define non isomorphic objects in $\barCM(A^\ltimes)$. 
		
		However, we can define two bijections that will preserve isomorphism classes :
		\begin{itemize}
			\item a rotation $\rho$ defined by $\rho(\overset{\to}{\gamma},p_0,J) = (\overset{\to}{\gamma},p_1,J^\varepsilon)$ where $p_1$ is the successor of $p_0$ in the decomposition of $\overset{\to}{\gamma}$, $\varepsilon=1$ if $\gamma_1$ and $\gamma_r$ have the boundary on the same side and $\varepsilon=-1$ otherwise.
			\item a symmetry $\sigma$ defined by $\sigma(\overset{\to}{\gamma},p_0,J) = (-\overset{\to}{\gamma},p_0,J^\varepsilon)$ where $-\overset{\to}{\gamma}$ is the band $\gamma$ with the other orientation, $\varepsilon=1$ if $\gamma_1$ and $\gamma_r$ have the boundary on the same side and $\varepsilon=-1$ otherwise.
		\end{itemize}
	\end{remark}
	
	\begin{proposition}
		If two parameters of $P(S,M,\Delta)$ are in the same orbit under the action of $\rho$ and $\sigma$ then the associated objects in $\barCM(A^\ltimes)$ are isomorphic. 
	\end{proposition}
	
	\begin{proof}
		Let $(\overset{\to}{\gamma},p_0,J)$ be a parameter. Let $\overset{\to}{\gamma} = \gamma_1 \dots \gamma_r$ be the associated decomposition. We define $\delta$ as follows 
		\[\delta = \begin{cases}
			1 \; \textup{if} \; f_r : P_{v_r} \to P_{v_0} \\
			-1 \; \textup{if} \; f_r : P_{v_0} \to P_{v_r}
		\end{cases}\]
		We now consider the isomorphism $\phi : M_{(\overset{\to}{\gamma},p_0,J)} \to M_{\rho(\overset{\to}{\gamma},p_0,J)}$ defined by $\phi = \sum_{i=1}^{r-1} \id_{P_{v_i}} I_n + \id_{P_{v_0}} J^\delta$.
		
		The way we defined $\delta$ ensures that it is a well defined isomorphism in $\barCM(A^\ltimes)$.
		
		Furthermore, if $\overset{\to}{\gamma} = \gamma_1 \dots \gamma_r$ is the decomposition with starting point $p_0$ then the decomposition of $-\overset{\to}{\gamma}$ with starting point $p_0$ is $-\overset{\to}{\gamma} = \gamma_r^{-1} \dots \gamma_1^{-1}$.
		
		We first recall that the map defined by the arc $\gamma_i$ does not depend on an orientation. By construction, we even have 
		\[M_{\sigma(\overset{\to}{\gamma},p_0,J)} \simeq M_{\rho(\overset{\to}{\gamma},p_0,J)}\]
		We can use the same isomorphism $\phi$ as before.
	\end{proof}
	
	\begin{corollary}
		Let $\gamma$ be a band of $(S,M,\Delta)$ and $(k ^n,J)$ an indecomposable $k[x,x^{-1}]$-module. There are at most two isomorphism classes of objects of $\barCM(A^\ltimes)$ defined as in Definition 2.3.3..
	\end{corollary}
	
	\begin{remark}
		It is not sufficient to give either starting point or an orientation, both are needed to properly define the object $M_{(\overset{\to}{\gamma},p_0,J)}$. However, giving a decomposition $\gamma = \gamma_1\dots \gamma_r$ fixes an orientation and a starting point. We will write $M_{(\gamma,J)}$ when a decomposition of $\gamma$ has been fixed. If a decomposition is not fixed, this notation may represent two different, non isomorphic objects.
	\end{remark}
	
	\begin{example}
		Considering the algebra as in Example 2.1.5., we can consider the closed curve homotopic to the boundary component. 
		
		
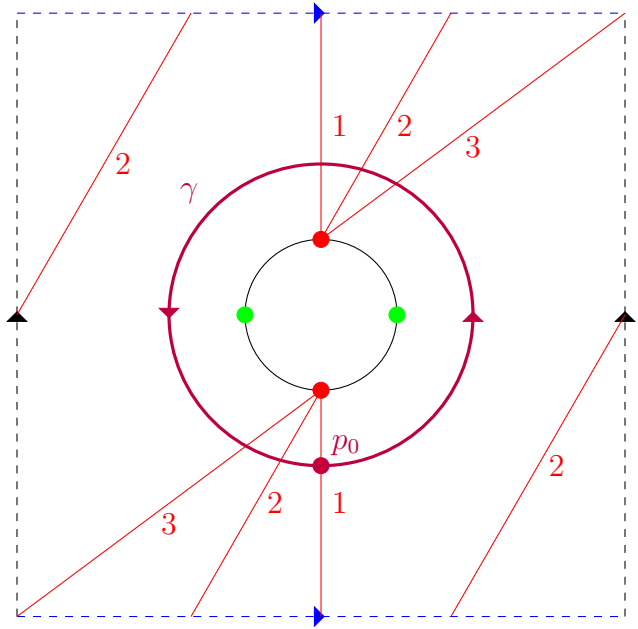
\begin{center}
			\begin{tikzpicture}
				\begin{scope}[on background layer]
					
					\draw (4,4) circle (1);
					\node [label = {[left, purple]:$\gamma$}] at (2.53,5.47) {};
					\node at (2,4) {\tikz \draw[-triangle 90,purple] (0,0) -- +(0,-0.1);};
					\node at (6,4) {\tikz \draw[-triangle 90,purple] (0,0) -- +(0,0.1);};
				\end{scope}
				
				\begin{scope}[dashed, every node/.style={sloped,allow upside down}]
					\draw[blue] (0,0)-- node {\midarrow} (8,0);
					\draw[black] (0,0)-- node {\midarrow} (0,8);
					\draw[blue] (0,8)-- node {\midarrow} (8,8);
					\draw[black] (8,0)-- node {\midarrow} (8,8);
				\end{scope}
				
				\begin{scope}[on above layer]
					\filldraw[red] (4,3) circle (3pt) node (1) {$ $};
					\filldraw[red] (4,5) circle (3pt) node (2) {$ $};
					\filldraw[green] (3,4) circle (3pt) node (3) {$ $};
					\filldraw[green] (5,4) circle (3pt) node (4) {$ $};
					\filldraw[purple] (4,2) circle (3pt) node[above right] (5) {$p_0$};

					\draw[red] (4,5) to node[right, midway] {$1$} (4,8);
					\draw[red] (4,3) to node[right, midway] {$1$} (4,0);
					
					\draw[red] (4,5) to node[below, midway] {$3$} (8,8);
					\draw[red] (4,3) to node[below, midway] {$3$} (0,0);
					
					\draw[red] (4,5) to node[right, midway] {$2$} (5.71,8);
					\draw[red] (5.71,0) to node[right, midway] {$2$} (8,4);
					\draw[red] (0,4) to node[right, midway] {$2$} (2.29,8);
					\draw[red] (2.29,0) to node[right, midway] {$2$} (4,3);
					
					\draw[purple, very thick] (4,4) circle (2);
				\end{scope}
			\end{tikzpicture}
		\end{center}
		The curve $\gamma$ does not admit any grading for the dissection $\Delta$. However, using the above definition, if $(k^n,J)$ is an indecomposable $k[x,x^{-1}]$-module, using the starting point and orientation as in the figure above, we have:
		\[M_{(\overset{\to}{\gamma},p_0,J)} = ((P_1 \oplus P_2 \oplus P_3)^{\oplus 2n}, a I_n + b I_n + ab' I_n + a' I_n + b' I_n + a'b J) \]
	\end{example}
	
	\begin{example}
		Let $A$ be the gentle algebra defined by the following quiver with relation.
		
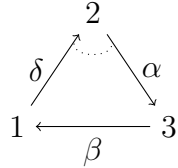
\begin{center}
			\begin{tikzpicture}
				\node (1) at (0,0) {$1$};
				\node (2) at (1,1.5) {$2$};
				\node (3) at (2,0) {$3$};

				\draw[->] (1) to node[left, midway] {$\delta$} (2);
				\draw[->] (2) to node[right, midway] {$\alpha$} (3);
				\draw[->] (3) to node[below, midway] {$\beta$} (1);
				\draw[dotted, bend right=90] (0.75,1.1) to (1.25,1.1);
			\end{tikzpicture}
		\end{center}
		
		The associated dissected marked surface is as follows :
		
		
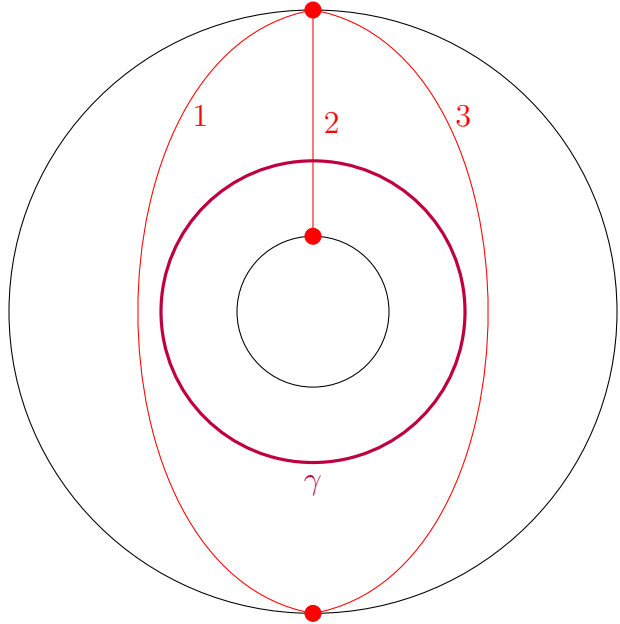
\begin{center}
			\begin{tikzpicture}
				\draw (0,0) circle (1);
				\draw (0,0) circle (4);
				
				\begin{scope}[on above layer]
					\filldraw[red] (0,1) circle (3pt) node (1) {$ $};
					\filldraw[red] (0,4) circle (3pt) node (2) {$ $};
					\filldraw[red] (0,-4) circle (3pt) node (3) {$ $};
					
					\draw[red] (0,1) to node[right, midway] {$2$} (0,4);
					\draw[red, bend right=80] (0,4) to node[right, near start] {$1$} (0,-4);
					\draw[red, bend left=80] (0,4) to node[right, near start] {$3$} (0,-4);
					
					\draw[purple, very thick] (0,0) circle (2);
					\node[purple] (4) at (0,-2.3) {$\gamma$};
				\end{scope}
			\end{tikzpicture}
		\end{center}
		Here, the closed curve $\gamma$ does not correspond to a band object in $\Dcat^b(\mod \, A)$ but it does correspond to a family generalized band object. If we consider the indecomposable $k[X]$-module $(k,\lambda)$ with $\lambda \in k^\times$ then the object $M_{(\gamma,\lambda)}$ is the $A^\ltimes$-module $(P_2, \lambda(\alpha \beta \delta))$ and corresponds to the following representation :
		
		
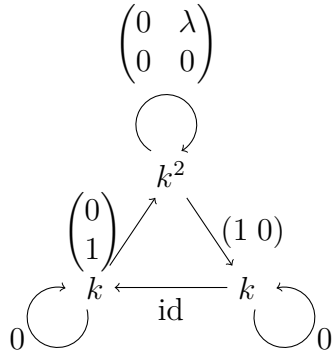
\begin{center}
			\begin{tikzpicture}
				\node (1) at (0,0) {$k$};
				\node (2) at (1,1.5) {$k^2$};
				\node (3) at (2,0) {$k$};

				\draw[->] (1) to node[left, midway] {$\begin{pmatrix}
						0 \\
						1
					\end{pmatrix}$} (2);
				\draw[->] (2) to node[right, midway] {$(1 \; 0)$} (3);
				\draw[->] (3) to node[below, midway] {$\id$} (1);
				\draw[->] (1.250) arc(15:-285:0.4)node[left,midway] {$0$};
				\draw[->] (3.290) arc(-195:105:0.4) node[right,midway] {$0$};
				\draw[->] (2.130) arc(240:-60:0.4) node[above, midway] {$\begin{pmatrix}
						0 & \lambda \\
						0 & 0
					\end{pmatrix}$};
				
			\end{tikzpicture}
		\end{center}
		
	\end{example}
	
	\subsection{Winding numbers in $\underbar{CM}(A^\ltimes)$}
	
	We will now see how to obtain the winding numbers of bands of $(S,M)$ by looking at the generalized band objects in $\underbar{CM}(A^\ltimes)$.
	
	First, we give a combinatorial definition of the winding number of a band. This definition was shown to be equivalent to the classical definition in \cite{APS}.
	
	\begin{definition}
		Let $\gamma$ be a band of $(S,M,\Delta)$ and $\gamma = \gamma_1 \dots \gamma_r$ be a decomposition with respect to $\Delta$. We define the winding number of $\gamma$ to be $w_\Delta(\gamma) = \displaystyle\sum_{i=1}^r \varepsilon_i$ where 
		\[\forall 1 \leq i \leq r, \; \varepsilon_i = \begin{cases}
			1 \; \textup{if the boundary is on the left of $\gamma_i$} \\
			-1 \; \textup{else}
		\end{cases}\]
	\end{definition}
	
	To recover the winding number of a band, we consider the family of automorphisms of $\barCM(A^\ltimes)$ defined in the following way.
	
	Let $\lambda$ be in $k^\times$, we have an automorphism of algebras :
	\begin{align*}
		t_\lambda : A^\ltimes &\to A^\ltimes \\
		(a,x) &\mapsto (a,\lambda x)
	\end{align*}
	This automorphism is the one given in \cite[Definition 4.1]{AO13} and that was also used in the proof of Proposition 1.4.1.
	
	This automorphism induces an autoequivalence of $\mod \, A^\ltimes$ which restricts to an autoequivalence of $\underbar{CM}(A^\ltimes)$. We denote $t_\lambda(M)$ the image of $M$ by the autoequivalence induced by $t_\lambda$.
	
	\begin{proposition}
		Let $A$ be a finite dimensional gentle algebra of finite global dimension and $(S,M,\Delta)$ its marked surface with dissection.
		\begin{itemize}
			\item Let $\gamma$ be a string of $(S,M)$, then $t_\lambda (M_\gamma) \simeq M_\gamma$.
			\item Let $\gamma = \gamma_1 \dots \gamma_r$ be a band of $(S,M)$ and $(k^n,J)$ an indecomposable $k[x,x^{-1}]$-module, then :
			\[t_\lambda(M_{(\gamma,J)}) \simeq M_{(\gamma, \lambda^{\varepsilon w_\Delta(\gamma)}J)} \]
			where $\varepsilon =1$ if $f_r : P_{v_{r-1}} \to P_{v_0}$ and $\varepsilon = -1$ otherwise.
		\end{itemize}
	\end{proposition}
	
	\begin{proof}
		Let $\gamma = \gamma_0\gamma_1 \dots \gamma_r\gamma_{r+1}$ be a string of $(S,M)$, we have $M_\gamma = (\bigoplus_{i=0}^{r} P_{v_i}, \sum_{i=1}^r f_i)$.
		
		\noindent Furthermore, if $\lambda$ is in $k^\times$, we have $t_\lambda(M_\gamma) = (\bigoplus_{i=0}^{r} P_{v_i}, \sum_{i=1}^r \lambda f_i)$.
		
		\noindent Consider now the following integers :
		
		\[\forall 1 \leq i \leq r, \; \varepsilon_i = \begin{cases}
			1 &\; \textup{if the boundary is to the left of $\gamma_i$} \\
			-1 &\; \textup{else}
		\end{cases}\]
		\[\forall 0 \leq i \leq r, \;  \eta_i = \sum_{j=1}^i \varepsilon_i \]
		We claim that the map $\phi = \sum_{i=0}^r \lambda^{- \eta_i} \id_{P_{v_i}}$ defines an isomorphism between $(M_\gamma)_\lambda$ and $M_\gamma$. To check it, it suffices to show 
		\[\left(\sum_{i=0}^r \lambda^{- \eta_i} \id_{P_{v_i}}\right) \circ \left(\sum_{i=1}^r \lambda f_i\right) \circ \left( \sum_{i=0}^r \lambda^{\eta_i} \id_{P_{v_i}}\right) = \sum_{i=1}^r f_i\]
		
		To do so, we look at the coeffecient in front of $f_i$ when doing the composition on the left. We have
		\begin{itemize}
			\item $\lambda^{\eta_{i-1} - \eta_i +1}$ when $\varepsilon_i = 1$ i.e. $f_i : P_{v_{i-1}} \to P_{v_i}$ ;
			\item $\lambda^{\eta_i - \eta_{i-1} +1}$ when $\varepsilon_i = -1$ i.e. $f_i : P_{v_i} \to P_{v_{i-1}}$.
		\end{itemize}
		
		In both cases, this coeffecient simplifies to 1.
		
		If $\gamma = \gamma_1 \dots \gamma_r$ is a band of $(S,M)$ and $(k^n,J)$ is an indecomposable $k[x,x^{-1}]$-module, we define the integers $(\varepsilon_i)_{1\leq i \leq r}$ and $(\eta_i)_{0 \leq i \leq r}$ as above. We claim that the map $\phi' = \sum_{i=0}^{r-1} \lambda^{- \eta_i} \id_{P_{v_i}^{\oplus n}}$ from $t_\lambda (M_{\gamma,J})$ to $M_{\gamma,\lambda^{\varepsilon w_\Delta(\gamma)}J}$ is an isomorphism. It is sufficient to show that it is well defined. To do so, we need only check 
		
		\[\left(\sum_{i=0}^{r-1} \lambda^{- \eta_i} \id_{P_{v_i}^{\oplus n}} \right) \circ \left(\sum_{i = 1}^{r-1} \lambda f_i I_n + \lambda f_r J \right) \circ \left(\sum_{i=0}^{r-1} \lambda^{- \eta_i} \id_{P_{v_i}^{\oplus n}} \right) = \sum_{i = 1}^{r-1} f_i I_n + \lambda^{\varepsilon w_\Delta(\gamma)} f_r J\]
		
		As before, to check this, we look at the coeffecient in front of each of the $f_i I_n$ for $1 \leq i \leq r-1$. The computation of the coeffecient is done in the same way as for the strings. We only need to check the coeffecient in front of $f_r J$. This coeffecient is 
		\begin{itemize}
			\item $\lambda^{\eta_{r-1} +1}$ if $\varepsilon_r = 1$ i.e. $f_r : P_{v_{r-1}} \to P_{v_0}$ ;
			\item $\lambda^{-\eta_{r-1} + 1}$ if $\varepsilon_r = -1$ i.e. $f_r : P_{v_0} \to P_{v_{r-1}}$.
		\end{itemize}
		These two cases can be summed up by the coeffecient being $\lambda^{\varepsilon w_\Delta(\gamma)}$ as $\varepsilon = \varepsilon_r.$
	\end{proof}
	
	\begin{remark}
		The autoequivalences $(t_\lambda)_{\lambda \in k^\times}$ define an action of $k^\times$ on the set of isomorphism classes of indecomposable objects of $\barCM(A^\ltimes)$. From Proposition 2.4.2. and \cite[Theorem 4.2]{AO14}, we recover the fact that the only band onjects in $\Dcat^b(A)$ are the one corresponding to closed curves with zero winding number.
	\end{remark}

	\section{Indecomposable objects of $\barCM(A^\ltimes$)}
	
	In this section, we prove that the objects of $\barCM(A^\ltimes)$ constructed in Section 2 are exactly the indecomposable objects of the category $\barCM(A^\ltimes)$ up to isomorphism.
	
	The proof completely follows the one in \cite{BM} which describes the indecomposable objects of $\Dcat^b(A)$ in terms of homotopy strings and homotopy bands. The key ingredient is a theorem of Bondarenko and Drozd (\cite{BD}) which gives a description of the indecomposable objects of a certain additive category $\Scat(\Ycat,\sigma,k)$ associated to a linearly ordered set with involution $(\Ycat,\sigma)$ in terms of strings and bands on a certain quiver.
	
	The strategy of the proof is as follows :
	
	\indent Given a gentle algebra $A$, we first define a linearly ordered set with involution $(\Ycat_A,\sigma_A)$ (Section 3.3) and then construct a functor $G : \CM^\textup{rad}(A^\ltimes) \to \Scat(\Ycat_A,\sigma_A,k)$ that sends non isomorphic indecomposable objects of $\CM^\textup{rad}(A^\ltimes)$ that are non projective to non isomorphic indecompasable objects of $\Scat(\Ycat_A,\sigma_A,k)$ (Proposition 3.3.5.).
	
	\indent Then we define an injective map $\gamma$ (resp. $\gamma_b$) from the set of homotopy strings (resp. homotopy bands) of $A$ to the set of strings (resp. bands) on the quiver defined by $(\Ycat_A,\sigma_A)$ (Lemma 3.3.8.).
	
	\indent We then show that $\gamma$ (resp. $\gamma_b$) is compatible with $G$ in the sense that

	\begin{multicols}{2}
		\begin{tikzpicture}
			\node (1) at (0,0) {$\overline{GSt(A)}$};
			\node (2) at (4,0) {$\overline{St(\Ycat_A)}$};
			\node (3) at (0,-2) {$\CM^\textup{rad}(A^\ltimes)$};
			\node (4) at (4,-2) {$\Scat(\Ycat_A,\sigma_A,k)$};
			
			\draw[->] (1) to node[above,midway] {$\gamma$} (2);
			\draw[->] (3) to node[below, midway] {$G$} (4);
			\draw[->] (1) to (3);
			\draw[->] (2) to (4);
		\end{tikzpicture}
		
		\begin{tikzpicture}
			\node (1) at (0,0) {$\overline{GBa(A)}$};
			\node (2) at (4,0) {$\overline{Ba(\Ycat_A)}$};
			\node (3) at (0,-2) {$\CM^\textup{rad}(A^\ltimes)$};
			\node (4) at (4,-2) {$\Scat(\Ycat_A,\sigma_A,k)$.};
			
			\draw[->] (1) to node[above,midway] {$\gamma_b$} (2);
			\draw[->] (3) to node[below, midway] {$G$} (4);
			\draw[->] (1) to (3);
			\draw[->] (2) to (4);
		\end{tikzpicture}
	\end{multicols}
	
	Finally we show (Lemma 3.3.9.) that the strings (resp. bands) in the image of $\gamma$ (resp. $\gamma_b$) are exactly the string objects (resp. band objects) in the image of $G$.
	
	The section will be organized as follows :
	\begin{itemize}
		\item First we recall from \cite{BM} the definition of $\Scat(\Ycat,\sigma,k)$ and the description of its indecomposable objects (Section 3.1).
		\item Then we follow \cite{OPS} to recall the definition of homotopy strings and homotopy bands on the quiver of a gentle algebra and how they relate to strings and bands on the associated surface (Section 3.2).
		\item Finally, we define the linearly ordered set with involution $(\Ycat_A,\sigma_A)$, construct the functor $G$ and prove the description of indecomposable objects of $\barCM(A^\ltimes)$ (Section 3.3).
	\end{itemize}
	
	\subsection{Bondarenko's representations of linearly ordered sets }
	
	Here, $\Ycat$ is a linearly ordered set and $\sigma$ an involution of $\Ycat$. In this section, all matrices will be with coefficients in $k$.
	
	\begin{definition}
		Let $B = (B_i^j)_{i,j \in \Ycat}$ be a finite block matrix, meaning that there is a finite number of non-zero blocks. Denote $l_i$ (resp. $c_j$) the number of rows (reps. columns) of the block $B_i^j$. Let $C = (C_i^j)_{i,j \in \Ycat}$ be another block matrix such that $C_i^j \in \Mcat_{l'_i,c'_j}(k)$. We say that the horizontal partition of $B$ is compatible with the vertical partition of $C$ when for every $i \in \Ycat$, $l_i = c'_i$.
		\\
		When this is the case, we can define $CB$ by block product.
	\end{definition} 
	
	\begin{definition}
		A $(\Ycat,\sigma)$-matrix is a finite block matrix $B = (B_i^j)_{i,j \in \Ycat}$ with coefficients in $k$ such that :
		\begin{itemize}
			\item The horizontal and vertical partitions of $B$ are compatible.
			\item If $i,j \in \Ycat$ are such that $\sigma(i) = j$ then the horizontal bands $B_i$ and $B_j$ have the same number of rows. Equivalently, the vertical bands $B^i$ and $B^j$ have the same number of columns.
			\item $B^2 = 0$
		\end{itemize}
	\end{definition}
	
	In order to define a category of $(\Ycat,\sigma)$-matrices, we need to define what a morphism between two $(\Ycat,\sigma)$-matrices is.
	
	\begin{definition}
		Let $B$ and $C$ be two $(\Ycat,\sigma)$-matrices. A morphism $T : B \to C$ is a block matrix $T = (T_i^i)_{i,j \in \Ycat}$ such that :
		\begin{itemize}
			\item The horizontal partition of $T$ is compatible with the vertical partition of $B$ and the vertical partition of $T$ is compatible with the horizontal partition of $C$ ;
			\item $TC = BT$ ;
			\item $T$ is upper triangular, i.e. if $j<i$ then $T_i^j = 0$ ;
			\item If $\sigma(i) = j$ then $T_i^i = T_j^j$.
		\end{itemize}
	\end{definition}
	
	We can now define the category $\Scat(\Ycat,\sigma,k)$ as the category with objects $(\Ycat,\sigma)$-matrices with coefficients in $k$ and morphisms those defined above. It is shown in \cite{B} that it is an additive $k$-linear category. Furthermore, the indecomposable objects of $\Scat(\Ycat,\sigma,k)$ are known and have an explicit description which we will now give.
	
	\begin{definition}
		We define the quiver $Q(\Ycat)$ as follows :
		\begin{itemize}
			\item the vertices are orbits of $\sigma$ : $Q(\Ycat)_0 = \Ycat / \sigma$, we will denote $[a]_\sigma$ the orbit of $a \in \Ycat$ ;
			\item the set of edges is $\Ycat \times \Ycat$ ;
			\item the source map is $s((a,b)) = [a]_\sigma$
			\item the target map is $t((a,b)) = [b]_\sigma$
		\end{itemize}
		We furthermore define two maps $p_1$ and $p_2$ to be the projections of $\Ycat \times \Ycat$ onto $\Ycat$ corresponding to the first component and the second component respectively.
		\\
		We also define the map $s : \Ycat^2 \to \Ycat^2$ that sends an arrow $(a,b)$ to its inverse arrow $(b,a)$.
	\end{definition}
	
	To describe indecomposable objects of $\Scat(\Ycat,\sigma,k)$ we will need to define strings and bands on $Q(\Ycat)$. However, these strings and bands are different from those we usually define on the quiver of a gentle algebra. We will therefore use a different notation.
	
	\begin{definition}
		We define $\text{St}(\Ycat)$ to be the set of paths $w = w_1 w_2 \dots w_n$ on the quiver $Q(\Ycat)$ such that : 
		\[ \forall i \in \{1,\dots , n-1\}, \; p_2(w_i) \neq p_1(w_{i+1})\]
		We extend the map $s$ to an involution sending a path to its inverse. We will denote $\overline{\textup{St}(\Ycat)}$ to be the set of equivalence classes of $\text{St}(\Ycat)$ under $s$. Its objects are called $\Ycat$-strings
	\end{definition}
	
	\begin{definition}
		We define $\textup{Ba}(\Ycat)$ to be the subset of $\text{St}(\Ycat)$ containing paths $w = w_1 \dots w_n$ such that $s(w_1) = t(w_n)$, $w^2 \in \textup{St}(\Ycat)$ and if $w = z^k$ then $k=1$ and $w=z$. 
		\\
		Consider the rotation map $r : Ba(\Ycat) \to Ba(\Ycat)$ sending a path $w = w_1 \dots w_n$ to the path $r(w) = w_2 \dots w_n w_1$. We will denote $\overline{\textup{Ba}(\Ycat)}$ to be the set of equivalence classes of $\textup{Ba}(\Ycat)$ under $s$ and $r$. Its elements are called $\Ycat$-bands
	\end{definition}
	
	We can now give the explicit description of indecomposable objects of $\Scat(\Ycat,\sigma,k)$.
	
	\begin{definition}
		Let $w = w_1 \dots w_n$ be a $\Ycat$-string. We define $c_j = t(w_j)$ which is in $\Ycat / \sigma$ for $j \in \{1,\dots,n\}$ and $c_0 = s(w_1)$ which is also in $\Ycat/ \sigma$. Let $V$ be a $k$-vector space of dimension $n+1$ and $v_0,\dots,v_n$ a basis of $V$. 
		\\
		For $x \in \Ycat$, we define $V_w([x]_\sigma) $ the subspace of $V$ generated by the $v_j$ such that $c_j = [x]_\sigma$.
		\\
		For $(x,y) \in Q(\Ycat)_1$ we define the linear map $V_w(x,y) : V_w([x]_\sigma) \to V_w([y]_\sigma)$ as follows 
		
		\[ V_w(x,y)(v_i) = \begin{cases} v_{i+1} \; \text{if} \;  w_{i+1} = (x,y) & \\ v_{i-1} \; \text{if} \;  w_i = (y,x) & \\ 0 \; \text{otherwise} \end{cases}\]
		
		We now define the $(\Ycat,\sigma)$-matrix $B_w$ as having the matrix of $V_w(x,y)$ in the fixed basis as block $(B_w)_x^y$. The number of rows of this block is $\# \{j \in \{0,\dots,n\} | t(w_j) = [y]_\sigma\}$ and the number of columns is $\# \{j \in \{0,\dots,n\} | t(w_j) = [x]_\sigma\}$.
	\end{definition}
	
	\begin{definition}
		We define in an analoguous way a $(\Ycat,\sigma)$-matrix for each band $w = w_1 \dots w_n$ equipped with an indecomposable $k[x,x^{-1}]$-module $(k^d,J)$. Let $\chi(t) = t^d + \sum_{i=0}^{d-1} a_i t^i$ be the characteristic polynomial of $J$. 
		\\
		Consider the vector space $V$ with basis $(v_{i,j})_{0 \leq i \leq n-1 ; 1 \leq j \leq d}$. For every $[x]_\sigma \in Q(\Ycat)_0$ we define the subspace $V_{w,J}([x]_\sigma)$ generated by the $v_{i,j}$ such that $c_j = [x]_\sigma$.
		\\
		For every $(x,y) \in Q(\Ycat)_1$, we define the linear map $V_{w,J}(x,y) : V_{w,J}([x]_\sigma) \to V_{w,J}([y]_\sigma)$ by 
		
		\[ V_{w,J}(x,y)(v_{i,j}) = \begin{cases}
			v_{i+1,j} &\; \text{if} \; i \neq n-1 \; \text{and} \; w_{i+1} = (x,y)  \\
			v_{i-1,j} &\; \text{if} \; i \neq 0 \; \text{and} \; w_i = (y,x) \\
			v_{0,j+1} &\; \text{if} \; i=n-1  , \, j \neq d \; \text{and} \; w_n = (x,y) \\
			v_{n-1,j+1} &\; \text{if} \; i=0  , \, j \neq d \; \text{and} \; w_n = (y,x) \\
			-\sum_{r=0}^{d-1} a_r v_{0,r} &\; \text{if} \; i=n-1 , \, j = d \; \text{and} \; w_n = (x,y) \\
			-\sum_{r=0}^{d-1} a_r v_{n-1,r} &\; \text{if} \; i=0  , \, j = d \; \text{and} \; w_n = (y,x) \\
			0 &\; \text{otherwise}
		\end{cases}\]
		We define the $(\Ycat,\sigma)$-matrix $B_{w,J}$ as having the matrix of $V_{w,J}(x,y)$ as block $(B_{w,J})_x^y$. The number of rows of this block is $d \times\# \{j \in \{0,\dots,n\} | t(w_j) = [y]_\sigma\}$ and the number of columns is $d \times \# \{j \in \{0,\dots,n\} | t(w_j) = [x]_\sigma\}$.
	\end{definition}
	
	\begin{remark}
		In this construction, the fifth and sixth cases cannot both happen for a same band. These cases, combined with the third and fourth ones (also only one of which will happen) make one of the blocks of $B_{w,J}$ have a sub-block that is the companion matrix of the polynomial $J$.
	\end{remark}
	
	\begin{remark}		
		Furthermore, in \cite{BM}, the parameter $J$ is an indecomposable polynomial in $k[x]$ that is not $x^d$. This has been changed here to an indecomposable $k[x,x^{-1}]$-module to ease the comparision with band objects in $\barCM(A^\ltimes)$. This change is purely esthetic as by considering a companion matrix or a characteristic polynomial, the two notions are identical.
	\end{remark}
	
	We now have the following theorem from \cite{BM}, combining \cite[Theorem~3]{BD} and \cite{NR}. 
	
	\begin{theorem}{\cite[Theorem~3]{BD}}
		Let $(\Ycat,\sigma)$ be a totally ordered set with involution. Let $B$ be an indecomposable object of $\Scat(\Ycat,\sigma,k)$, then one of the following occurs :
		\begin{itemize}
			\item there exists a unique $\Ycat$-string $w$ such that $B \simeq B_ w$ ;
			\item there exists a unique $\Ycat$-band $w$ and an indecomposable $k[x,x^{-1}]$-module $(k ^d,J)$ such that $B \simeq B_{w,J}$. 
		\end{itemize}
	\end{theorem}
	
	We will use this description of indecomposables of $ \Scat(\Ycat,\sigma,k)$ by constructing an additive functor from a subcategory of $\textup{CM}(A^\ltimes)$ to $ \Scat(\Ycat,\sigma,k)$.
	
	\subsection{Homotopy strings and bands}
	
	We recall from \cite{OPS} the definition of homotopy strings and bands and the relation with strings and bands on the surface of $A$.
	
	Homotopy strings and bands are in bijection with the strings and bands of the associated graded surface $(S,M,\Delta)$ of the gentle algebra. 
	
	\begin{definition}
		Let $(Q,I)$ be a quiver with relations, $Q_1$ the set of arrows of $Q$ and $Q_0$ the set of vertices. For each arrow $\alpha \in Q_1$, we define its formal inverse $\alpha^{-1}$. Let $Q_1^{-1}$ be the set of formal inverses of all arrows in $Q_1$. We ask that this formal inverse to be an involution, such that $s(\alpha^{-1}) = t(\alpha)$ and that $t(\alpha^{-1}) = s(\alpha)$.
		\\
		A walk on $(Q,I)$ is a sequence $w = w_1\dots w_n$ where each $w_i$ lies in $Q_1$ or $Q_1^{-1}$. Furthermore, we ask that $t(w_i) = s(w_{i+1})$. We also consider trivial walks to be walks of length zero, we denote them $e_i$ where $i \in Q_0$ is its source and target.
		\\
		A string of $(Q,I)$ is a walk $w = w_1\dots w_n$ where for all $i \in \{1,\dots,n-1\}$, $w_{i+1} \neq w_i^{-1}$ and if $w_i \dots w_j$ is a subwalk of $w$ where all arrows lie in $Q_1$ (resp. $Q_1^{-1}$), then $w_i \dots w_j \notin I$ (resp. $(w_i \dots w_j)^{-1} \notin I$).
		\\
		If all $(w_i)_{1 \leq i \leq n}$ are in $Q_1$ (reps. $Q_1^{-1}$), we say that $w$ is a direct (resp. inverse) string.
	\end{definition}
	
	We can now define finite homotopy strings and homotopy bands.
	
	\begin{definition}
		Let $(Q,I)$ be the quiver of a gentle algebra $A$. A finite homotopy string is a walk $w$ that can be written as the concatenation of subwalks $(w_i)_{1 \leq i \leq r}$ such that :
		\begin{itemize}
			\item each $w_i$ is a direct or inverse string ;
			\item if $w_i$ and $w_{i+1}$ are both direct (resp.inverse) strings then $w_i w_{i+1} \in I$ (resp. $(w_i w_{i+1})^{-1} \in I $).
		\end{itemize}
		A homotopy band is a finite homotopy string $w$ such that $w^2$ is a finite homotopy string and if $w = z^k$ for $z$ a finite homotopy string then $k = 1$ and $z = w$. 
		\\
		A homotopy string or homotopy band is said to be reduced if for all $i \in \{1,\dots,r-1\}$, $w_{i+1} \neq w_i^{-1}$.
	\end{definition}
	
	We denote $GSt(A)$ (resp. $GBa(A)$) the set of reduced finite homotopy strings (resp. reduced homotopy bands) on $(Q,A)$.
	\\
	We will now define a map :
	\begin{align*}
		\tilde{\kappa} : GSt(A) &\to St(S,M) \\
		w = w_1\dots w_r &\mapsto \tilde{\kappa}(w).
	\end{align*} 
	Here, $\tilde{\kappa}(w)$ is the isotopy class of a curve constructed as follows :
	\begin{itemize}
		\item consider $\gamma_i$ the arc in a polygon of $(S,M,\Delta)$ joining two midpoints of arcs of $\Delta$ defined by $w_i$ which is a direct or inverse string of $A$ ;
		\item define $\tilde{\gamma}$ to be the curve defined by connecting the arcs $\gamma_i$, it is well defined because $t(w_i) = s(w_{i+1})$ ;
		\item at these endpoints, we can extend this curve to the other polygon of which $s(w)$ and $t(w)$ are the tags of the edges, so that the end points of this curve are the marked points in these polygons. We obtain an arc $\tilde{\kappa}(w)$ joining two marked points of $(S,M,\Delta)$.
	\end{itemize}	
	
	Because $w_{i+1}\neq w_i^{-1}$, all intersections of $\tilde{\kappa}(w)$ with $\Delta$ are transverse. 
	
	We define in the same way the map :
	\begin{align*}
		\tilde{\kappa_b} : GBa(A) &\to Ba(S,M) \\
		w = w_1\dots w_r \mapsto \tilde{\kappa_b}(w)
	\end{align*}
	
	This map is defined in an analog way to $\kappa$, but because $t(w) = s(w)$, the arcs in the polygons concatenate to form a band of $(S,M,\Delta)$.
	
	\begin{definition}
		The map $s : w \mapsto w^{-1}$ defines an involution of $GSt(A)$. We denote $\overline{GSt(A)}$ the set of equivalence classes of $GSt(A)$ under $s$.
		\\
		This map also defines an involution of $GBa(A)$. If we consider $r : w = w_1 \dots w_r \mapsto w_2 \dots w_r w_1$, it defines automorphism of $GBa(A)$. We denote $\overline{GBa(A)}$ the set of equivalence classes of $GBa(A)$ under $s$ and $r$.
	\end{definition}

	The map $\tilde{\kappa}$  (resp. $\tilde{\kappa_b}$) induces a map $\overline{GSt(A)} \to St(S,M)$ (resp. $\overline{GBa(A)} \to Ba(S,M)$) that we will denote $\kappa$ (resp. $\kappa_b$).
	
	\begin{proposition}{\cite[Lemma 2.15, Lemma 2.16]{OPS}}
		The maps $\kappa$ and $\kappa_b$ are bijections.
	\end{proposition}

	\subsection{Full description of indecomposable objects of $\underbar{CM}(A^\ltimes)$}
	
	Let $\Mcat$ be the set of maximal paths on the quiver of $A$ i.e. paths that are not subpaths of a path of strictly greater length. For $w \in \Mcat$, we consider the following set :
	
	\[\Ycat_w := \{u | u \; \textup{subpath} \; \textup{of} \; w , s(u) = s(w)\}\]
	
	\noindent For a path $w$, we define $l(w)$ to be its length. We define an order on $\Ycat_w$ by
	
	\[\forall u,v \in \Ycat_w, \; u \leq_w v \iff l(u) \leq l(v)\]
	
	\noindent This is a total order on $\Ycat_w$ and $w$ is the maximal element. 
	
	\noindent Let us fix a total order $\leq_\Mcat$ on $\Mcat$, we can now consider the set :
	\[\Ycat = \bigcup_{w \in \Mcat} \Ycat_w\]
	We define an order on $\Ycat$ by 
	\[\forall u \in \Ycat_w, \; \forall v \in \Ycat_{w'}, \; u \leq v \iff w <_\Mcat w' \,\textup{or}\, (w=w' \,\textup{and}\, u \leq_w v) \]
	
	The following lemma enables us to define an involution on $\Ycat$.
	
	\begin{lemma}[\cite{BM}]
		Let $A$ be a gentle algebra with quiver $(Q,I)$. For every $i \in Q_0$, there are at most two paths in $\Ycat$ with target $i$.
	\end{lemma}
	
	We now define an involution $\sigma$ on $\Ycat$ as follows :
	
	\[\forall u \in \Ycat, \; \sigma(u) = \begin{cases}
		v \; \textup{if there exists a path $v$ in $\Ycat \backslash \{u\}$ such that $t(v) = t(u)$} \\
		u \; \textup{else}
	\end{cases} \]
	
	The aim of this subsection is to construct an additive functor $G : \CMrad(A^\ltimes) \to \Scat(\Ycat,\sigma,k)$. 
	
	We first define $G$ on objects. Let $M = (P,\varphi)$ be an object of $\CMrad(A)$ and $P = \bigoplus_{i \in Q_0} P_i^{d_i}$ be its decomposition into indecomposable projective $A$-modules. 
	
	We define $\Pcat$ to be the set of paths on the quiver of $A$. We also define $\Pcat_0$ to be the set of trivial paths on the quiver of $A$. If $w : i \to j$ is a path in $\Pcat$, we denote $\varphi_w : P_i \to P_j$ the associated morphism in $\mod \, A$.
	
	We have a decomposition $\varphi = \displaystyle\sum_{w \in \Pcat} \varphi_w A_w$ where $A_w \in \Mcat_{d_j,d_i}(k)$ is the matrix representing the multiplicities of $\varphi_w$.
	
	Because $M$ is in $\CMrad(A^\ltimes)$, if $w$ is a trivial path then $A_w = 0$. We can therefore write 
	\[\varphi = \displaystyle\sum_{w \in \Pcat \backslash \Pcat_0} \varphi_w A_w\]
	
	For every non-trivial path $\alpha$ of $A$, there is a unique maximal path $w$ such that $w = u\alpha \tilde{w}$. We can thus define the block matrix $G(M) := \left(G(M)_u^v\right)_{u,v \in \Ycat}$ as follows :
	
	\[\forall u,v \in \Ycat, \; G(M)_u^v = \begin{cases}
		A_\alpha^\mathsf{T} & \; \text{if} \; v = u \alpha \\
		0 & \; \text{otherwise}
	\end{cases} \in \Mcat_{d_{t(u)},d_{t(v)}}(k)\]
	
	\begin{lemma}
		If $M \in \CMrad(A^\ltimes)$ then $G(M)$ is a $(\Ycat,\sigma)$-matrix.
	\end{lemma}
	
	\begin{proof}
		First, we show that the horizontal and vertical partitions of $G(M)$ are compatible. Let $u$ be in $\Ycat$, in the horizontal band $G(M)_u$, every matrix has $d_{t(u)}$ rows and and in the vertical band $G(M)^u$, every matrix has $d_{t(u)}$ columns. Thus horizontal and vertical partitions are compatible.
		\\
		Secondly, if $\sigma(u) = v$ then $u$ and $v$ have the same target. As the number of rows in $G(M)_u$ (resp. $G(M)_v$) is $d_{t(u)}$ (reps. $d_{t(v)}$), the second condition is also verified.
		\\
		Lastly, we need to show that $G(M)^2 = 0$. We have :
		\[0 = \varphi^2 = \sum_{w \in \Pcat} \sum_{\substack{\alpha,\beta \in \Pcat\\ \alpha\beta = w}} \varphi_\beta \circ  \varphi_\alpha A_\beta A_\alpha = \sum_{w \in \Pcat} \varphi_w \left(\sum_{\substack{\alpha,\beta \in \Pcat\\ \alpha\beta = w}} A_\beta A_\alpha\right)\]
		
		Therefore we get $\displaystyle\sum_{\substack{\alpha,\beta \in \Pcat\\ \alpha\beta = w}} A_\beta A_\alpha = 0$ for every $w$ in $\Pcat$.
		\\
		Let $u$ and $v$ be in $\Ycat$. We have 
		\[(G(M)^2)_u ^v = \sum_{x \in \Ycat} G(M)_u^x G(M)_x^v\]
		The term $G(M)_u^x G(M)_x^v$ is always zero except when $x = u\alpha$ and $v=x\beta$ where $\alpha$ and $\beta$ are paths. In this case, we get $v = uw$ and thus 
		\[(G(M)^2)_u ^v = \sum_{\substack{\alpha,\beta \in \Pcat\\ \alpha\beta = w}} A^\mathsf{T}_\alpha A^\mathsf{T}_\beta = (\sum_{\substack{\alpha,\beta \in \Pcat\\ \alpha\beta = w}} A_\beta A_\alpha)^\mathsf{T} = 0\]
	\end{proof}
	
	We now define $G$ on morphisms. Let $M = (P,\varphi)$ and $N = (Q,\varphi')$ be in $\CMrad(A^\ltimes)$. A morphism $\psi :M \to N$ is the data of a morphism $\psi$ in $\Hom_A(P,Q)$ such that $\psi \circ \varphi = \varphi' \circ \psi$. 
	\\
	Let $P = \bigoplus_{i \in Q(A)_0} P_i^{d_i}$ and $Q = \bigoplus_{i \in Q(A)_0} P_i^{d'_i}$ be the decompositions of $P$ and $Q$ into indecomposable modules. A morphism $\psi \in \Hom_A(P,Q)$ decomposes as follows :
	\[\psi = \sum_{w \in \Pcat(A)} \varphi_w A_w\]
	where $A_w \in \Mcat_{d'_j,d_i}(k)$ if $w : i \to j$.
	\\
	We define the morphism $G(\psi) : G(M) \to G(N)$ as follows :
	\[\forall u,v \in \Ycat, \; G(\psi)_u^v = \begin{cases}
		A_\alpha^\mathsf{T} &\; \text{if} \; v = u\alpha \\
		0 & \; \text{otherwise}
	\end{cases} \in \Mcat_{d_i,d'_j}(k)\]
	
	\begin{lemma}
		If $\psi : M \to N$ then $G(\psi)$ is in $\Hom_{\Scat(\Ycat,\sigma,k)}(G(M),G(N))$.
	\end{lemma}
	
	\begin{proof}
		Let $u$ be in $\Ycat$, the blocks in $G(\psi)_u$ have $d'_{t(u)}$ rows and blocks in $G(\psi)^u$ have $d_{t(u)}$ columns. Because blocks in $G(M)^u$ have $d_{t(u)}$ columns and blocks in $G(N)_u$ have $d'_{t(u)}$ rows, the horizontal (resp. vertical ) partition of $G(\psi)$ is compatible with the vertical partition of $G(M)$ (resp. the horizontal partition of $G(N)$).
		\\
		If $M = (P,\varphi)$, $N = (Q,\varphi')$ and $\psi : M \to N$ we have decompositions 
		
		\begin{multicols}{3}
			\noindent 
			\begin{equation*}
				\varphi = \sum_{w \in \Pcat} \varphi_w M_w
			\end{equation*}
			\noindent
			\begin{equation*}
				\varphi' = \sum_{w \in \Pcat} \varphi_w N_w
			\end{equation*}
			\noindent
			\begin{equation*}
				\psi = \sum_{w \in \Pcat} \varphi_w A_w
			\end{equation*}
		\end{multicols} 
		
		The equality $\psi \circ \varphi_M = \varphi_N \circ \psi$ implies that for all $w \in \Pcat$ 
		\[ \sum_{w = w_1 w_2} A_{w_2}M_{w_1} = \sum_{w = w_1 w_2} N_{w_2} A_{w_1} \]
		Furthermore, if $u,v \in \Ycat$, such that $v = u \alpha$ we have 
		\[(G(M)G(\psi))_u^v = \sum_{w \in \Ycat} G(M)_u^w G(\psi)_w^v = \sum_{\alpha = \beta \delta} {}^t M_\beta {}^t A_\delta\]
		and 
		\[(G(\psi)G(N))_u^v = \sum_{\alpha = \beta \delta} {}^t A_\beta {}^t N_\delta\]
		Thus we have $G(M)G(\psi) = G(\psi)G(N)$.
		\\
		Next, $G(\psi)$ is upper triangular because $G(\psi)_u^v$ is non-zero implies that $u$ is a subpath of $v$. Given the ordering on $\Ycat$, this implies that $u \leq v$.
		\\
		Finally, if $\sigma(u) = v$ then we have $G(\psi)_u^u = G(\psi)_v^v$ by definition of $G(\psi)$.
		
	\end{proof}
	
	\begin{remark}
		Using the same method, it is straightforward to see that $G$ defines an additive functor.
	\end{remark}
	
	\begin{proposition}
		Let $(\Ycat,\sigma)$ be the linearly ordered set with involution as defined above. Then we have the following :
		\begin{itemize}
			\item[(a)] $\forall M \in \CMrad(A^\ltimes), \; G(M) = 0 \iff M \in \proj \, A^\ltimes$.
			\item[(b)] If $\mathcal{U}$ is the full subcategory of $\Scat(\Ycat,\sigma,k)$ consisting of objects in the image of $G$ then
			\[X \simeq Y \; \textup{in} \; {\mathcal{U}} \iff X \simeq Y \; \textup{in} \; \Im(G)\]
		\end{itemize}
	\end{proposition}
	
	\begin{proof}
		This proof will follow very closely the proof of an analog result in \cite{BM}.
		\begin{itemize}
			\item[(a)] Let $M = (P,\varphi) \in \CMrad(A^\ltimes)$ such that $G(M) = 0$. Then if $\varphi = \displaystyle\sum_{w \in \Pcat} \varphi_w A_w$, we have $A_w = 0$ for all $w \in \Pcat$. Thus $\varphi = 0$ and $M \in \proj A$.
			\item[(b)] The implication from right to left is obvious. To prove the implication from right to left, we construct a functor $H : \Ucat \to \Im(G)$ that is the identity on the objects and such that if $T$ is a morphism in $\Ucat$ : 
			\[\forall u,v \in \Ycat, H(T)_u^v = \begin{cases}
				T_u^v & \; \textup{if} \; v = u \alpha \\
				0 & \; \textup{otherwise}
			\end{cases}\]  
			
			By construction, $H(\varphi)$ is always a morphism of $\Im(G)$. We only need to show that it respects composition of morphisms. Let $T$ and $T'$ be two morphisms of $\Ucat$ such that $T' \circ T$ is well defined. Let $u$ and $v$ be in $\Ycat$. We have
			\[(H(T' \circ T))_u^v =\begin{cases}
				\displaystyle\sum_{x \in \Ycat} (T')_u^x T_x^v & \; \textup{if} \; v = uw \\
				0 & \; \textup{otherwise}
			\end{cases} \]
			
			When $v = uw$, for the term $(T')_u^x T_x^v$ to be non-zero, it is necessary that $u \leq x$ and $x \leq v$. Because $u$ and $v$ have the same source, there exists a maximal path $\delta$ such that $u$ and $v$ are in $\Ycat_\delta$. Thus by definition of the order on $\Ycat$, we have that $x$ is in $\Ycat_\delta$. Therefore there exists two paths $\alpha$ and $\beta$ such that $x = u \alpha$ and $v = x \beta$. This means that $w = \alpha \beta$ and we have
			
			\[(H(T' \circ T))_u^v  = \begin{cases}
				\displaystyle\sum_{\substack{\alpha,\beta \in \Pcat(A)\\ \alpha\beta = w}} (T')_u^{u\alpha} T_{u\alpha}^v & \; \textup{if} \; v = uw \\
				0 & \; \textup{otherwise}
			\end{cases}\]
			
			However, we have :
			\[(H(T') \circ H(T))_u^v = \begin{cases}
				\displaystyle\sum_{x \in \Ycat} H(T')_u^x H(T)_x^v & \; \textup{if} \; v = uw \\
				0 & \; \textup{otherwise}
			\end{cases}\]
			
			For the term $H(T')_u^x H(T)_x^v$ to be non-zero, it is necessary that $x = u\alpha$ and $y = x\beta$. In this case, we have $H(T')_u^x = (T')_u^x$ and $H(T)_x^v = T_x^v$. We also have 
			\[(H(T') \circ H(T))_u^v = \begin{cases}
				\displaystyle\sum_{\substack{\alpha,\beta \in \Pcat(A)\\ \alpha\beta = w}} (T')_u^{u\alpha} T_{u\alpha}^v & \; \textup{if} \; v = uw \\
				0 & \; \textup{otherwise}
			\end{cases}\]
			
			Thus we have $H(T' \circ T) = H(T') \circ H(T)$ and $H$ is a functor.
			
			Therefore it transforms an isomorphism between $X$ and $Y$ in $\Ucat$ to an isomorphism between $X$ and $Y$ in $\Im(G)$.
		\end{itemize}
	\end{proof}
	
	We now have all the tools to give and prove the main theorem of this section.
	
	\begin{theorem}
		The indecomposable objects of $\underbar{CM}(A^\ltimes)$ are exactly those of the form :
		\begin{itemize}
			\item $M_y$ where $y$ is a string of $(S,M,\Delta)$ ;
			\item $M_{(y,J)}$ where $y = y_1 \dots y_r$ is a band of $(S,M,\Delta)$ and $(k^n,J)$ is an indecomposable $k[x,x^{-1}]$-module.
		\end{itemize}
	\end{theorem}
	
	This proof is organized in four steps and follows the one in \cite{BM}. The first step is the construction of two maps $\gamma$ and $\gamma_b$ that relate strings on the quiver of $A$ to strings on $Q(\Ycat)$ and the same for bands. The second gives a characterization of the images of $\gamma$ and $\gamma_b$. The last step is checking the compatibility between $G$ and these maps.
	
	\begin{definition}
		If $A = kQ/I$ is a gentle algebra and $w$ is a non trivial path of $(Q,I)$ then there exists a unique maximal path $\tilde{w}$ having $w$ as a subpath. We will use the notations $\tilde{w} = \hat{w}w\bar{w}$.
	\end{definition}
	
	\begin{lemma}
		The map $\gamma : GSt(A) \to St(\Ycat(A),\sigma)$ defined by : 
		\begin{itemize}
			\item if $w$ is a direct arrow of $A$ then $\gamma(w) = (\hat{w},\hat{w}w)$ ;
			\item if $w$ is an inverse arrow of $A$ then $\gamma(w) = (\hat{w^{-1}}w^{-1}, \hat{w^{-1}})$ ;
			\item if $w = w_1 \dots w_r$ is a homotopy string of $A$ then $\gamma(w) = \gamma(w_1)\dots \gamma(w_r)$ ;
		\end{itemize}
		is well defined, and induces an injective map $\overline{GSt(A)} \to \overline{St(\Ycat,\sigma)}$ which we will also denote $\gamma$.
		
		Furthermore, if we define $\gamma_b : GBa(A) \to Ba(\Ycat,\sigma)$ to be the restriction of $\gamma$ to $GBa(A)$ then this map is well defined and induces an injective map $\overline{GBa(A)} \to \overline{Ba(\Ycat,\sigma)}$ also denoted by $\gamma_b$
	\end{lemma}
	
	\begin{proof}
		We now need to show that $\gamma$ is well defined. Let $w = w_1 \dots w_n \in GSt(A)$ and $i \in \{1,\dots,n-1\}$. It is sufficient to show that $t(\gamma(w_i)) = s(\gamma(w_{i+1}))$ and $p_2(\gamma(w_i)) \neq p_1(\gamma(w_{i+1}))$. There are four possible cases :
		\begin{itemize}
			\item If $w_i,w_{i+1}$ are direct arrows. In this case, $w_iw_{i+1} \in I$, $t(\gamma(w_i)) = [\hat{w_i}w_i]_\sigma$ and $s(\gamma(w_{i+1})) = [\hat{w_{i+1}}]_\sigma$. To show that these two equivalence classes are equal, we need to show the representatives have the same target.
			
			We have $t(\hat{w_{i+1}}) = s(w_{i+1}) = t(w_i) = t(\hat{w_i}w_i)$ which is what we wanted.
			
			Furthermore, $p_1(\gamma(w_{i+1})) = \hat{w_{i+1}} = e_{s(w_{i+1})} \neq \hat{w_i} w_i = p_2(\gamma(w_i))$.
			\item If $w_i$ and $w_{i+1}$ are both inverse arrows, the reasoning is analog to the previous case.
			\item If $w_i$ is a direct arrow and $w_{i+1}$ is an inverse arrow, then $w_i$ and $w_{i+1}^{-1}$ have same target but do not share any arrows.
			
			In this case, $t(\gamma(w_i)) = [\hat{w_i}w_i]_\sigma$ and $s(\gamma(w_{i+1})) = [\hat{w_{i+1}^{-1}}w_{i+1}^{-1}]_\sigma$. However $t(\hat{w_i}w_i) = t(w_i) = t(w_{i+1}^{-1}) = t(\hat{w_{i+1}^{-1}}w_{i+1}^{-1})$ thus the equivalence classes under $\sigma$ are the same.
			
			Furthermore, $p_1(\gamma(w_{i+1})) = \hat{w_{i+1}^{-1}}w_{i+1}^{-1} \neq \hat{w_i}w_i = p_2(\gamma(w_i))$.
			\item If $w_i$ is an inverse arrows and $w_{i+1}$ is a direct arrow, then this case is analog to the previous one.
		\end{itemize}
		
		If $w = w_1 \dots w_n$ and $v = v_1 \dots v_m$ are elements of $GSt(A)$ such that $\gamma(v) = \gamma(w)$ then $n = m$ and for all $i \in \{1,\dots,n\}$, $\gamma(v_i) = \gamma(w_i)$. However, if this is true then $\hat{v_i} = \hat{w_i}$ and $\hat{v_i}v_i = \hat{w_i}w_i$ which means that $v_i = w_i$. Therefore $\gamma$ is injective. 
		
		\noindent The map $\gamma$ induces an injective map $\overline{GSt(A)} \to \overline{St(\Ycat(A),\sigma)}$ which we will also denote $\gamma$.
		
		\noindent We can show in an analog way that $\gamma_b$ is well defined and injective. As before, $\gamma_b$ induces a map $\gamma_b : \overline{GBa(A)} \to \overline{Ba(\Ycat,\sigma)}$ that is also injective.
	\end{proof}
	
	\begin{lemma}
		If $y$ is in $\overline{GSt(A)}$ then $G(M_{\kappa(y)}) \simeq B_{\gamma(y)}$ where $\kappa$ is the bijection of Proposition 2.2.4. .If $z$ is in $\overline{GBa(A)}$ and  $(k^n,J)$ is an indecomposable $k[x,x^{-1}]$-module, then $G(M_{\kappa_b(y),J}) \simeq B_{\gamma_b(z),f}$ where $\kappa_b$ is as in Proposition 2.2.4..
	\end{lemma}
	
	\begin{proof}
		Let $y = y_1 \dots y_r$ be in $\overline{GSt(A)}$, define $v_i = t(y_i)$ and $v_0 = s(y_1)$. Then $M_{\kappa(y)} = (\bigoplus_{i=0}^r P_{v_i}, \varphi)$ with $\varphi = \sum_{i=1}^r f_i$ where $f_i$ is the map induced by the path (or antipath) $y_i$. The multiplicity of these morphisms is always one.
		
		We have $\gamma(y_i) = (\hat{y_i},\hat{y_i} y_i)$ if $y_i$ is a path and $\gamma(y_i) = (\hat{y_i^{-1}} y_i^{-1},\hat{y_i^{-1}})$ if it is an antipath. The contribution of $\gamma(y_i)$ in $B_{\gamma(y)}$ appears in the block $(B_{\gamma(y)})_{\hat{y_i}}^{\hat{y_i}y_i}$ or $(B_{\gamma(y)})_{\hat{y_i^{-1}}y_i^{-1}}^{\hat{y_i^{-1}}}$ depending on if it is a path or an antipath. This block is the one corresponding to the path $y_i$ (if it is a path) or $y_i^{-1}$ (if $y_i$ is an antipath).
		
		This contribution is exactly the one of the $f_i$ in $G(M_{\kappa(y)})$, meaning that we do have $G(M_{\kappa(y)}) = B_{\gamma(y)}$.
		
		If $y = y_1 \dots y_r$ is in $\overline{GBa(A)}$, we can use the same argument, with the following difference : the multiplicities of the $f_i$ are not all one. If $i \neq r$, then the multiplicity is the $d\times d$ identity matrix and for $f_r$, the multiplicity is given by companion matrix of $J$. These multiplicites are exactly what is described in the description of $B_{\gamma(y),F}$. 
		
	\end{proof}
	
	\begin{lemma}
		If $\gamma$ and $\gamma_b$ are as in lemma 3.3.7 then we have
		\[\Im(\gamma) = \{w \in \overline{St(\Ycat,\sigma)} | B_w \in \Im(G)\}\]
		and 
		\[\Im(\gamma_b) = \{w \in \overline{Ba(\Ycat,\sigma)} | B_{w,1} \in \Im(G)\}\]		
	\end{lemma}
	
	\begin{proof}
		In this section we will first show :
		\[\Im(\gamma) = \{w \in St(\Ycat,\sigma) | B_w \in \Im(G)\}\]
		
		First, we show the inclusion from right to left.
		\\
		Let $w  \in St(\Ycat,\sigma)$ and $w = w_1 \dots w_n$ a decomposition of the string, such that $B_w \in \Im(G)$.
		\\
		If $B_w$ is in the image of $G$ then all non-zero blocks are in position $(B_w)_u^v$ with $v = u\alpha$. Therefore, by definition of $B_w$, for all $i \in \{1,\dots,n\}$, there exists a non-trivial path $x_i \in \Pcat(A)$ such that $w_i = (\hat{x_i},\hat{x_i}x_i)$ (1) or $w_i = (\hat{x_i}x_i,\hat{x_i})$ (2). We then define :
		\[y_i = \begin{cases}
			x_i &\; \textup{in the case (1)}\\
			x_i^{-1} &\; \textup{in the case (2)}\\
		\end{cases}\]
		
		We now have $w_i = (\hat{y_i},\hat{y_i}y_i)$ or $w_i = (\hat{y_i^{-1}}y_i^{-1},\hat{y_i^{-1}})$. Define $y = y_1\dots y_n$, we will now show that $y \in GSt(A)$.
		
		Let $i \in \{1,\dots,n-1\}$, it is sufficient to show that $y_i y_{i+1} \in GSt(A)$. There are four cases :
		\begin{itemize}
			\item If $y_i$ and $y_{i+1}$ are both paths, we have to show that they compose and that $y_i y_{i+1} \in I$. In this case, we have $w_i = (\hat{y_i},\hat{y_i}y_i)$ and $w_{i+1} = (\hat{y_{i+1}},\hat{y_{i+1}}y_{i+1})$.
			
			As $t(w_i) = s(w_{i+1})$, we have $[\hat{y_i}y_i]_\sigma = [\hat{y_{i+1}}]_\sigma$ i.e. $t(y_i) = t(\hat{y_{i+1}}) = s(y_i)$.
			
			As $p_2(w_i) \neq p_1(w_{i+1})$ we have $\hat{y_i}y_i \neq \hat{y_{i+1}}$. By definition $\hat{y_{i+1}}y_{i+1} \neq 0$, we  have $\hat{y_i}y_i y_{i+1} = 0$ because $A$ is gentle. Thus $y_i y_{i+1} = 0$ by definition of $\hat{y_i}$. 
			\item If $y_i$ and $y_{i+1}$ are antipaths, we use analog arguments to the previous case.
			\item If $y_i$ is a path and, $y_{i+1}$ is an antipath, then $w_i = (\hat{y_i},\hat{y_i}y_i)$ and $w_{i+1} =  (\hat{y_{i+1}^{-1}}y_{i+1}^{-1}, \hat{y_{i+1}^{-1}})$. 
			
			In this case, $[\hat{y_i}y_i]_\sigma = t(w_i) = s(w_{i+1}) = [\hat{y_{i+1}^{-1}}y_{i+1}^{-1}]_\sigma$. This means that $t(y_i) = t(y_{i+1}^{-1}) = s(y_{i+1})$. 
			
			As $p_2(w_i) \neq p_1(w_{i+1})$ we have $\hat{y_i}y_i \neq \hat{y_{i+1}^{-1}} y_{i+1}^{-1}$. Because they have the same target, by uniqueness of maximal paths, $y_i$ and $y_{i+1}^{-1}$ do not share any arrow and thus $y_i y_{i+1} \in \overline{GSt(A)}$.
			\item If $y_i$ is an antipath and $y_{i+1}$ is a path, we use the same arguments as in the previous case.
		\end{itemize}
		Thus we have found $y \in GSt(A)$ such that $w = \gamma(y)$. This gives the inclusion :
		\[  \{w \in \overline{St(\Ycat,\sigma)} | B_w \in \Im(G)\} \subset \Im(\gamma) \]
		To show the other inclusion, it is sufficient to see that if $y \in \overline{GSt(A)}$, we have
		\[ G(M_{\kappa(y)}) \simeq B_{\gamma(y)} \]
		
		We now show the following characterization for the image of $\gamma_b$ :
		\[\Im(\gamma_b) = \{w \in \overline{Ba(\Ycat,\sigma)} | B_{w,1} \in \Im(G)\}\]
		The arguments for the inclusion form right to left are exactly the same than for $\Im(\gamma)$ and the argument for the other inclusion is that if $y \in \overline{GBa(A)}$ then :
		\[G(M_{\kappa_b(y),1}) \simeq B_{\gamma_b(y),1}\]
	\end{proof}
	
	We now have the necessary tools to prove theorem 3.3.6.
	
	\begin{proof}{\textit{of Theorem 3.3.6}}
		
		Let $M$ be an indecomposable object in $\CMrad(A)$, then $G(M)$ is an indecomposable object in $\Scat(\Ycat,\sigma,k)$. There are two possibilities by Theorem 3.1.11. :
		\begin{itemize}
			\item There exists $w \in \overline{St(\Ycat,\sigma)}$ such that $G(M) \simeq B_w$. Then by Lemma 3.2.10., there exists $y \in \overline{GSt(A)}$ such that $w = \gamma(y)$. Thus by Lemma 3.2.9., $G(M) \simeq G(M_{\kappa(y)})$. However, by Proposition 3.2.5. this isomorphism can be seen as an isomorphism in $\Im(G)$, meaning that $M \simeq M_{\kappa(y)}$.
			\item There exists $w \in \overline{Ba(\Ycat,\sigma)}$ and $(k^n,J)$ an indecomposable $k[x,x^{-1}]$-module such that $G(M) \simeq B_{w,J}$. Thus by Lemma 3.2.10. there exists $y \in \overline{GBa(A)}$ such that $w = \gamma_b(y)$ and $G(M) \simeq B_{\gamma_b(y), f} \simeq G(M_{\kappa_b(y),f})$. By the same argument as above, we have $M \simeq M_{\kappa_b(y),f}$.
		\end{itemize}
		
	\end{proof}
	
	\bibliographystyle{alpha}
	\bibliography{biblio}
	
	\begin{center}
		\textbf{Joseph Winspeare}
		\\ Address : Institut Fourier, 100 rue des maths, 38610 Gières, Université Grenoble Alpes
		\\ email : joseph.winspeare@univ-grenoble-alpes.fr
	\end{center}
	
	\end{document}